\documentclass[a4paper]{amsart}

\title{Eberlein oligomorphic groups}

\author{Ita{\"\i} Ben Yaacov}
\address{Univ Lyon \\Universit\'e Claude Bernard Lyon 1 \\Institut Camille Jordan, CNRS UMR 5208 \\43 boulevard du 11 novembre 1918\\69622 Villeurbanne \textsc{Cedex} \\France}

\author{Tom\'as Ibarluc{\'\i}a}
\address{Univ Lyon \\Universit\'e Claude Bernard Lyon 1 \\Institut Camille Jordan, CNRS UMR 5208 \\43 boulevard du 11 novembre 1918\\69622 Villeurbanne \textsc{Cedex} \\France}
\curraddr{
  Institut de Math{\'e}matiques de Jussieu--PRG \\
  Universit\'e Paris 7, case 7012 \\
  75205 Paris \textsc{cedex} 13 \\
  France
}
\email{ibarlucia@math.univ-paris-diderot.fr}

\author{Todor Tsankov}
\address{
  Institut de Math{\'e}matiques de Jussieu--PRG \\
  Universit\'e Paris 7, case 7012 \\
  75205 Paris \textsc{cedex} 13 \\
  France}
\email{todor@math.univ-paris-diderot.fr}

\subjclass[2010]{Primary: 22F50. Secondary: 03C45, 22A25, 03C98, 22A15.}
\keywords{Hilbert compactification, oligomorphic, $\aleph_0$-categorical, Fourier--Stieltjes algebra, semitopological semigroup compactification, inverse semigroup, Hilbert-representable}

\usepackage{amsmath, amssymb, amsthm, amscd, enumerate}

\usepackage[T1]{fontenc}
\usepackage{kpfonts}
\usepackage{a4wide}

\usepackage{stmaryrd}

\makeatletter
\g@addto@macro\bfseries{\boldmath}
\makeatother

\def\mbZ{\mathbb{Z}}
\def\mbQ{\mathbb{Q}}
\def\mcA{\mathcal{A}}
\def\mbC{\mathbb{C}}
\def\mcH{\mathcal{H}}
\def\UH{U(\mathcal{H})}
\def\ThH{\Theta(\mathcal{H})}
\def\CG{C(G)}

\DeclareMathOperator{\Aut}{Aut}
\DeclareMathOperator{\Iso}{Iso}
\DeclareMathOperator{\Homeo}{Homeo}
\DeclareMathOperator{\Hilb}{Hilb}
\DeclareMathOperator{\WAP}{WAP}
\DeclareMathOperator{\RUC}{RUC}
\DeclareMathOperator{\UC}{UC}
\DeclareMathOperator{\acl}{acl}
\DeclareMathOperator{\Cb}{Cb}
\DeclareMathOperator{\tp}{tp}
\DeclareMathOperator{\rank}{rank}
\newcommand{\eq}{\mathrm{eq}}

\theoremstyle{plain}        \newtheorem{fact}{Fact}[section]
\theoremstyle{plain}        \newtheorem{theorem}[fact]{Theorem}
\theoremstyle{plain}        \newtheorem{lem}[fact]{Lemma}
\theoremstyle{plain}        \newtheorem{prop}[fact]{Proposition}
\theoremstyle{plain}        \newtheorem{cor}[fact]{Corollary}
\theoremstyle{definition}   \newtheorem{rem}[fact]{Remark} 
\theoremstyle{definition}   \newtheorem{defin}[fact]{Definition}
\theoremstyle{definition}   \newtheorem{convention}[fact]{Convention}
\theoremstyle{definition}   
\theoremstyle{definition}   \newtheorem{question}[fact]{Question}
\theoremstyle{definition}   \newtheorem{example}[fact]{Example}

\numberwithin{equation}{section}

\usepackage[hidelinks]{hyperref}

\def\Ind#1#2{#1\setbox0=\hbox{$#1x$}\kern\wd0\hbox to 0pt{\hss$#1\mid$\hss}
\lower.9\ht0\hbox to 0pt{\hss$#1\smile$\hss}\kern\wd0}
\def\ind{\mathop{\mathpalette\Ind{}}}
\def\Notind#1#2{#1\setbox0=\hbox{$#1x$}\kern\wd0\hbox to 0pt{\mathchardef
\nn="3236\hss$#1\nn$\kern1.4\wd0\hss}\hbox to 0pt{\hss$#1\mid$\hss}\lower.9\ht0
\hbox to 0pt{\hss$#1\smile$\hss}\kern\wd0}

\usepackage{xcolor}

\newcommand{\cl}[2][]{\overline{#2}^{#1}}
\newcommand{\actson}{\curvearrowright}

\begin{document}

\begin{abstract}
We study the Fourier--Stieltjes algebra of Roelcke precompact, non-archimedean, Polish groups and give a model-theoretic description of the Hilbert compactification of these groups. We characterize the family of such groups whose Fourier--Stieltjes algebra is dense in the algebra of weakly almost periodic functions: those are exactly the automorphism groups of $\aleph_0$-stable, $\aleph_0$-categorical structures. This analysis is then extended to all semitopological semigroup compactifications $S$ of such a group: $S$ is Hilbert-representable if and only if it is an inverse semigroup. We also show that every factor of the Hilbert compactification is Hilbert-representable.
\end{abstract}

\maketitle
\tableofcontents

\section*{Introduction}
It has long been recognized in model theory that the action of the automorphism group of an $\aleph_0$-categorical structure on the structure (and its powers) captures all model-theoretic information about it. Moreover, by a classical result of Ahlbrandt and Ziegler~\cite{ahlzie84}, the automorphism group remembers the structure up to bi-interpretability. As most interesting model-theoretic properties are preserved by interpretations, it is reasonable to expect that those would correspond to natural properties of the automorphism group.

It turns out that many model-theoretic properties of the structure are reflected in the behaviour of a certain universal dynamical system associated to the group that we proceed to describe. First, recall that automorphism groups of $\aleph_0$-categorical structures are Roelcke precompact in the following sense.

\begin{defin}
A topological group $G$ is called \emph{Roelcke precompact} if for every neighbourhood $U$ of the identity, there exists a finite set $F$ such that $UFU = G$.
\end{defin}

To each Roelcke precompact Polish group $G$ one can naturally associate its \emph{Roelcke compactification} $R(G)$, the completion of $G$ with respect to its Roelcke (or lower) uniformity; see Subsection~\ref{The WAP compactification...} for more details. The natural action $G \actson R(G)$ renders it a topological dynamical system. From the model-theoretic point of view, if we represent $G$ as the automorphism group of an $\aleph_0$-categorical structure $M$, $R(G)$ can be considered as a suitable closed subspace of the type space $S_\omega(M)$ in infinitely many variables over the model. Thus, there is a natural correspondence between formulas with parameters from the model, on the one hand, and continuous functions on $R(G)$, on the other. This allows building a dictionary between the model-theoretic and the dynamical setting. For example, \emph{stable} formulas correspond to \emph{weakly almost periodic} (WAP) functions and \emph{NIP} formulas correspond to \emph{tame} functions.

Particularly relevant to us is the theory of Banach representations of dynamical systems as developed by Glasner and Megrelishvili in a series of papers (see \cite{glamegSurvey} and the references therein). If $G \actson X$ is a topological dynamical system and $V$ is a Banach space, a \emph{representation of $X$ on $V$} is a pair of continuous maps $\iota \colon X \to B$, $\pi \colon G \to \Iso(V)$, where $B$ is the unit ball of $V^*$ equipped with the weak$^*$ topology, $\Iso(V)$ is the group of linear isometries of $V$, equipped with the strong operator topology, $\pi$ is a homomorphism, and
\begin{equation*}
\langle v, \iota(gx) \rangle = \langle \pi(g)^{-1}v, \iota(x) \rangle,
\end{equation*}
for all $x \in X$, $v \in V$, $g \in G$ (here, $\langle v,\varphi\rangle=\varphi(v)$ is the usual pairing of $V$ and $V^*$). A representation is \emph{faithful} if $\iota$ is an embedding. If $\mathcal{K}$ is a class of Banach spaces, we say that $G \actson X$ is \emph{$\mathcal{K}$-representable} if it admits a faithful representation on a Banach space in the class $\mathcal{K}$.

All dynamical systems are representable on some Banach space; however, if one restricts to some (well-chosen) class of Banach spaces $\mathcal{K}$, the $\mathcal{K}$-representable systems usually form an interesting family. Somewhat unexpectedly, in the $\aleph_0$-categorical setting, there are some precise connections between model-theoretic properties of the structure and the classes of Banach spaces $R(G)$ can be represented on: for example, $M$ is stable iff $R(G)$ can be represented on a reflexive Banach space \cite[\textsection 5]{bentsa}\cite[\textsection 5.1]{glamegSurvey} and $M$ is NIP iff $R(G)$ can be represented on a Banach space that does not contain a copy of $\ell^1$ \cite[\textsection 4]{iba14}\cite[\textsection 8.1]{glamegSurvey}. One of the main motivating questions for this paper was what the appropriate model-theoretic condition is for $R(G)$ to be representable on a Hilbert space.

For some classes $\mathcal{K}$ of Banach spaces, there are dynamical systems that are universal for the $\mathcal{K}$-representable ones. For example, $W(G)$, the \emph{WAP compactification} of $G$, is universal for reflexively representable systems, and $H(G)$, the \emph{Hilbert compactification}, is universal for Hilbert-representable systems. Both $W(G)$ and $H(G)$ carry the structure of a compact \emph{semitopological semigroup} and $H(G)$ is a factor of $W(G)$.

The main focus of this paper are the automorphism groups of $\aleph_0$-categorical \emph{classical, discrete} (multi-sorted) structures or, equivalently, Roelcke precompact, Polish, \emph{non-archimedean} groups. (A~group is \emph{non-archimedean} if it admits an open basis at the identity consisting of open subgroups.) We make this assumption tacitly throughout the paper: when we say ``$\aleph_0$-categorical structure'', we will always mean a classical one, as opposed to metric. A non-archimedean, Polish, Roelcke precompact group will be called \emph{pro-oligomorphic}; it is \emph{oligomorphic} if the structure can be chosen one-sorted.

For every non-archimedean group $G$, the compactification $G\to H(G)$ is a topological embedding. Our first result is a concrete description of $H(G)$ for pro-oligomorphic groups, in model-theoretic terms. More precisely, we have the following.
\begin{theorem}
Let $M$ be an $\aleph_0$-categorical structure and let $G = \Aut(M)$. Then $H(G)$ is isomorphic to the semigroup of partial elementary embeddings $M^\eq \to M^\eq$ with algebraically closed domains.
\end{theorem}

Using this description, we give two characterizations of pro-oligomorphic groups for which $W(G) = H(G)$: one model-theoretic, and one in terms of the semigroup $W(G)$. This is the main result of the paper.
\begin{theorem}
  \label{th:i:main}
Let $M$ be an $\aleph_0$-categorical structure and let $G = \Aut(M)$. The following are equivalent:
\begin{enumerate}
\item The idempotents of $W(G)$ commute.
\item $M$ is one-based for stable independence.
\item $W(G) = H(G)$.
\end{enumerate}
\end{theorem}

Using Theorem~\ref{th:i:main} and a classical, deep result in model theory, we can now give a satisfactory answer to our initial question.
\begin{cor}
  \label{c:omega-stable}
  Let $M$ be an $\aleph_0$-categorical structure and let $G = \Aut(M)$. Then the following are equivalent:
  \begin{enumerate}
  \item $M$ is $\aleph_0$-stable.
  \item $R(G)$ is Hilbert-representable.
  \end{enumerate}
\end{cor}

Corollary~\ref{c:omega-stable} and a well-known example of an $\aleph_0$-categorical, stable, non-$\aleph_0$-stable structure, due to Hrushovski, give us the following corollary (cf. Example~\ref{ex:Hrushovski}) which answers a question of Glasner and Megrelishvili \cite[Question~6.10]{glamegSurvey}.
\begin{cor}
  \label{c:Hrushovski}
  There exists an oligomorphic group $G$ that satisfies $R(G) = W(G) \neq H(G)$.
\end{cor}

While all factors of $W(G)$ are known to be reflexively representable (or \emph{reflexively approximable}, for a general topological group $G$), it is an open question whether all factors of $H(G)$ are Hilbert-representable \cite[Question~5.12.3]{glamegSurvey}. We can give a positive answer to this question in the case of pro-oligomorphic groups (cf. Theorem~\ref{theorem factors of H(G)}).

\begin{theorem}
  Let $G$ be a pro-oligomorphic group. Then all factors of $H(G)$ are Hilbert-representable.
\end{theorem}

The correspondence between model-theoretic properties of $\aleph_0$-categorical structures and dynamical properties of their automorphism groups is not restricted to the non-archimedean case. The correct model-theoretic setting for dealing with general Roelcke precompact Polish groups is that of continuous logic and in both \cite{bentsa} and \cite{iba14}, the results are proved in full generality. However, the two most important tools used in this paper are currently only available in the non-archimedean setting: namely, the classification of the unitary representations on the dynamical side and the notion of one-basedness on the model-theoretic side. For the moment, we do not even have a plausible conjecture of what the model-theoretic characterization of Hilbert-representable functions on a Roelcke precompact Polish group should be in general. Theorem~\ref{th:i:main} clearly fails in the continuous setting (for example, for the unitary group). While we do not have a counterexample to Corollary~\ref{c:omega-stable} for general separably categorical structures, we strongly suspect that it also fails.

As one of the goals of this paper is to provide a dictionary between model theory and abstract topological dynamics, we have tried to make the exposition fairly self-contained (apart from a couple of difficult model-theoretic results) and accessible to people working in both areas.

\medskip \noindent \textbf{Acknowledgements.}
Part of the present work was carried out during the trimester program \emph{Universality and Homogeneity} organized by the Hausdorff Research Institute for Mathematics in Bonn in 2013. We would like to thank the institute as well as the organizers of the program for the excellent conditions and stimulating atmosphere they provided. We are grateful to Michael Megrelishvili for valuable comments on a preliminary draft and to an anonymous referee for a careful reading of the paper and many comments that helped us improve the presentation.

All three authors were partially supported by the ANR contract GrupoLoco (ANR-11-JS01-0008). T.I. was also partially supported by the ANR contract ValCoMo (ANR-13-BS01-0006). T.T. was also partially supported by the ANR contract GAMME (ANR-14-CE25-0004).

\noindent\hrulefill

\section{Preliminaries}

\subsection{Compactifications of topological groups} Let $G$ be a topological group. The algebra of complex-valued continuous bounded functions on $G$ will be denoted by $\CG$. This algebra always carries the uniform norm, $\|f\|=\sup_{g\in G}|f(g)|$. The group $G$ admits a left and a right action on $\CG$, given, respectively, by $(gf)(h)=f(g^{-1}h)$ and $(fg)(h)=f(hg^{-1})$ for every $f\in\CG$ and $g,h\in G$. These actions are isometric but in general not continuous.

When considering subalgebras of $\CG$, we will always assume that these are unital and closed under complex conjugation. If we say that a subalgebra is \emph{closed}, we mean closed with respect to the uniform norm. \emph{Left-invariant}, \emph{right-invariant} and \emph{bi-invariant} refer to the actions of $G$ defined above.

A \emph{compactification} of $G$ is a compact Hausdorff space $X$ equipped with a continuous left action of $G$, together with a continuous $G$-map $\alpha \colon G\to X$ with dense image (where $G$ carries the natural left action on itself). This is the same as choosing a point $x_0 \in X$ with a dense orbit; then one can simply define $\alpha(g) = g \cdot x_0$. Such a pair $(X, x_0)$ is often called a \emph{$G$-ambit} in the literature.

To every compactification $\alpha \colon G\to X$ we associate the algebra $\mcA(\alpha):=C(X)\circ\alpha$ consisting of those functions in $\CG$ that factor through $\alpha$. The algebra $\mcA(\alpha)$ is always left-invariant, and the compactification will be called \emph{bi-invariant} if $\mcA(\alpha)$ is also right-invariant.

Given two compactifications $\alpha_X \colon G\to X$ and $\alpha_Y \colon G\to Y$, we say that $\alpha_Y$ is a \emph{$G$-factor} of $\alpha_X$ (or simply that $Y$ is a factor of $X$) if there is a continuous surjective $G$-map $\pi \colon X\to Y$ such that $\alpha_Y = \pi \circ \alpha_X$. If such a factor map exists, it is always unique.

A tool that we will use throughout the paper is \emph{Gelfand duality}: the contravariant equivalence between the category of compact Hausdorff spaces with continuous maps and the category of commutative, unital $\mathrm{C}^*$-algebras with algebra homomorphisms which is given by the functors $X \mapsto C(X)$ and $A \mapsto \hat A$. Here, $\hat A\coloneqq\operatorname{Hom}(A, \mbC)$ is the compact space of unital algebra homomorphisms $A\to\mbC$ endowed with the topology of pointwise convergence, and is called the \emph{Gelfand space} of $A$. In particular, one can identify $A$ with $C(\hat A)$. See, for example, \cite[Chapter~1]{Folland1995} for details. This is similar to the duality between Boolean algebras and their Stone spaces, which is perhaps more familiar to logicians.

We will be particularly interested in algebras of the form $\mcA(\alpha)$ for some compactification $\alpha$ of $G$. Among them there is a maximal one, the algebra $\RUC(G)$ of right uniformly continuous, bounded functions. (A function $f\in\CG$ is \emph{right uniformly continuous} if the orbit map $g\in G\mapsto gf\in\CG$ is norm-continuous.) If $\alpha\colon G\to X$ is a compactification of $G$, then $\mcA(\alpha) \cong C(X)$ is a left-invariant, closed subalgebra of $\RUC(G)$. Conversely, if $A$ is a left-invariant, closed subalgebra of $\RUC(G)$, then $X=\hat A$ is equipped with a continuous $G$-action (namely, $(gx)(f)=x(g^{-1}f)$), and the natural map $\alpha\colon G\to X$ (given by $\alpha(g)(f)=f(g)$) is a compactification of $G$, which satisfies $\mcA(\alpha)=A$.
We list below some facts and conventions regarding this duality that we use throughout the paper:
\begin{itemize}
\item When considering compactifications, we may omit the map $\alpha$ and refer only to the space~$X$ if no confusion arises. In particular, we may write $\mcA(X)$ instead of $\mcA(\alpha)$. 
 
\item A compactification $X_1$ is a factor of $X_2$ iff $\mcA(X_1) \subset \mcA(X_2)$. If $\mcA(X_1) \subset \mcA(X_2)$, then, under the identification $X_i=\operatorname{Hom}(\mcA(\alpha_i), \mbC)$, the factor map $\pi\colon X_2\to X_1$ is simply the restriction of homomorphisms.

\item (Stone--Weierstrass) Let $A_0\subset A\subset\CG$, where $A$ is a closed subalgebra and $A_0$ is any set. Then, $A$ is the closed subalgebra generated by $A_0$ iff $A_0$ separates the points of $\hat A$ (here we make the usual identification of $A$ with $C(\hat A)$).

\item A closed subalgebra $A\subset C(G)$ is separable if and only if $\hat A$ is metrizable.
\end{itemize}

\subsection{The Fourier--Stieltjes algebra and the WAP algebra of a topological group} Recall that if $\mcH$ is a Hilbert space, its unitary group $\UH$ equipped with the strong operator topology (pointwise convergence) is a topological group.

\begin{defin} A \emph{(unitary) matrix coefficient} of a topological group $G$ is a function $f\in\CG$ of the form $$f(g)=\langle v,\pi(g)w\rangle = \langle \pi(g)^{-1}v, w \rangle,$$ where $\pi \colon G\to\UH$ is a continuous unitary representation and $v,w\in\mcH$. We will use the notation $f=m_{v,w}$, or $f=m^\pi_{v,w}$ if we wish to specify $\pi$.\end{defin}

By considering orthogonal sums, tensor products, and duals of representations, one sees that the matrix coefficients of a topological group $G$ form a subalgebra of $\CG$. The family of all matrix coefficients of $G$ is called the \emph{Fourier--Stieltjes algebra of $G$}, and is denoted by $B(G)$.

Next we recall the definition of weakly almost periodic functions, Grothendieck's double limit criterion, and the reflexive representation theorem of Megrelishvili. (See \cite{groCriteres}, Th\'eor\`eme~6, and \cite{megFragmentability}, Theorem~5.1.)

\begin{defin}\label{def:wap}
A function $f\in\CG$ is \emph{weakly almost periodic} if the following equivalent conditions hold:
\begin{enumerate}
\item The orbit $Gf$ is precompact (i.e., has compact closure) for the weak topology of $\CG$.
\item For all sequences $g_i,h_j\in G$, the following limits coincide whenever they exist: $$\lim_i\lim_jf(g_ih_j)=\lim_j\lim_if(g_ih_j).$$
\item There exists a continuous, isometric representation $\pi \colon G\to\Iso(V)$ on a reflexive Banach space $V$ and vectors $v\in V$, $w\in V^*$ such that, for all $g\in G$, $$f(g)=\langle \pi(g)^{-1}v, w\rangle.$$
\end{enumerate}
\end{defin}

It follows easily that the family $\WAP(G)$ of weakly almost periodic functions on $G$ is a closed, bi-invariant subalgebra of $\RUC(G)$ containing $B(G)$. On the other hand, $B(G)$ is almost never closed in $\CG$ (see the beginning of Section~\ref{B(G) of olig groups}). Following \cite[\textsection 6]{glamegSurvey}, we will denote the closure $\overline{B(G)}$ by $\Hilb(G)$.
The algebra $B(G)$ is bi-invariant, hence so is $\Hilb(G)$.

Thus we have $$\Hilb(G)\subset\WAP(G),$$ or, equivalently: the \emph{Hilbert compactification} $G\to H(G)$ associated to the closed left-invariant algebra $\Hilb(G)$ is a $G$-factor of the \emph{WAP compactification} $G\to W(G)$ associated to $\WAP(G)$. We will review the main properties of these compactifications in Section~\ref{Semitop semigroup comp}.

Finally, we recall that a function $f \colon G\to\mbC$ is \emph{Roelcke uniformly continuous} if the map $(g,g')\in G\times G\mapsto gfg'\in\CG$ is norm-continuous. The family of all Roelcke uniformly continuous functions on $G$ is a closed, bi-invariant subalgebra of $\RUC(G)$, denoted by $\UC(G)$. We always have $\WAP(G)\subset\UC(G)$ (see, for instance, \cite{glamegSurvey}, Theorem~3.19).

\begin{defin}
Let $G$ be a topological group.
\begin{enumerate}[(i)]
\item $G$ is \emph{Eberlein} if $\Hilb(G)=\WAP(G)$.
\item $G$ is a \emph{WAP group} if $\WAP(G)=\UC(G)$.
\item $G$ is \emph{strongly Eberlein} if $\Hilb(G)=\UC(G)$.
\end{enumerate}
\end{defin}

In his fundamental work \cite{eberlein49}, Eberlein introduced weakly almost periodic functions (in the context of locally compact abelian groups) and proved the inclusion $B(G)\subset\WAP(G)$. In fact, all his examples of WAP functions lay in the closure of $B(G)$. Rudin writes in \cite{rudinWeak} that Eberlein asked him whether the closure of $B(G)$ may in fact coincide with $\WAP(G)$. Of course, by the Peter--Weyl theorem, this is the case for compact groups (indeed, $\Hilb(G)=\CG$). However, Rudin showed that this is not true in general. As an example, he exhibited a concrete function $f\in\WAP(\mbZ)\setminus\Hilb(\mbZ)$. Later, Chou~\cite{chou82} proved that the inclusion $\Hilb(G)\subset\WAP(G)$ is strict for any non-compact, locally compact, nilpotent group. On the other hand, he remarked that equality does hold for some non-compact, locally compact groups, and introduced the name \emph{Eberlein} for this class. The definitions of \emph{WAP groups} and \emph{strongly Eberlein groups} were introduced by Glasner and Megrelishvili in \cite{glamegSurvey}.

Examples of non-compact Eberlein groups include $\operatorname{SL}_n(\mathbb{R})$ (and any semisimple Lie group with finite centre; see \cite{veech79}), the unitary group $U(\ell^2)$ (essentially \cite{Uspenskij1998}), the group $\Aut(\mu)$ of measure-preserving transformations of the unit interval \cite{glasnerAutmu}, and the symmetric group of a countable set, $S_\infty$ \cite{glamegSurvey}. The last three are in fact strongly Eberlein. We will give some new examples in Subsection~\ref{Eberlein pro-oligomorphic}.

\subsection{Representations on Hilbert spaces}

Let $X$ be a compactification of a Polish group $G$. We say that $X$ is \emph{Hilbert-representable} if there exist a Hilbert space $\mcH$, an embedding $\iota\colon X\to \mcH$ (where $\mcH$ carries the weak topology) and a unitary representation $\pi\colon G\to\UH$ such that
$\iota(gx) = \pi(g)\iota(x)$
for all $x\in X$ and $g\in G$. By the Riesz representation theorem, this definition coincides, for the class of Hilbert spaces, with the notion of $\mathcal{K}$-representability given in the introduction.

Given a function $f\in\RUC(G)$, let $X_f$ be the compactification of $G$ associated to the left-invariant closed subalgebra of $\RUC(G)$ generated by $f$. In \cite[\textsection 2]{glaweiHilbert}, it is observed that $X_f$ is Hilbert-representable whenever $f$ is \emph{positive definite}, in the case $G=\mbZ$; more generally, the following holds.

\begin{lem} If $f\in B(G)$, then $X_f$ is Hilbert-representable.\end{lem}
\begin{proof}
Write $f=m^{\pi_0}_{v_0,w_0}$ for some continuous unitary representation $\pi_0 \colon G\to U(\mcH_0)$. Let $\mcH_1$ be the closed linear span of $\pi(G)w_0$ and let $v$ be the orthogonal projection of $v_0$ to $\mcH_1$. Next let $\mcH$ be the closed linear span of $\pi(G)v$ and let $w$ be the orthogonal projection of $w_0$ to $\mcH$. Consider the restriction $\pi = \pi_0|_\mcH$. Then $f=m^\pi_{v,w}$.

Recall that the unit ball of a Hilbert space is weakly compact; hence, if we let $Z$ be the weak closure of $\pi(G)w$ in $\mcH$, this is naturally a (Hilbert-representable) compactification of $G$ via the map $g\mapsto\pi(g)w$. Consider for each $h\in G$ the function $F_h\in C(Z)$, $F_h(z)=\langle\pi(h)v,z\rangle$, and note that $F_h(\pi(g)w)=hf(g)$. Since $\mcH$ is generated by $\pi(G)v$, we have $z=z'$ in $Z$ iff for every $h$ we have $\langle\pi(h)v,z-z'\rangle=0$, i.e., $F_h(z)=F_h(z')$; that is, the functions $F_h$ separate points of $Z$. Hence, by the Stone--Weierstrass theorem, $\mcA(Z)$ is the closed algebra generated by $Gf$. In other words, $Z$ is isomorphic to $X_f$.
\end{proof}

In contrast, if instead of a matrix coefficient we take any $f\in\Hilb(G)$, it is unknown whether $X_f$ is necessarily Hilbert-representable; see Question \ref{question factors of H(G)} below.

\begin{prop}\label{Equiv Hilb-rep} Let $\alpha\colon G\to X$ be a metrizable compactification of $G$. Then the following are equivalent:
\begin{enumerate}
\item $\alpha$ is Hilbert-representable.
\item $\mcA(\alpha)=\overline{\mcA(\alpha)\cap B(G)}$.
\end{enumerate}
\end{prop}
\begin{proof}
Suppose $(\iota,\pi)$ is a representation of $(X,G)$ on a Hilbert space $\mcH$. The functions $F_v\colon w\mapsto\langle v,w\rangle$ separate points of $\mcH$, hence the algebra generated by $\{F_v\iota\alpha\}_{v\in\mcH}$ is dense in $\mcA(\alpha)$ and contained in~$B(G)$.

Conversely, suppose that (2) holds. The metrizability assumption on $X$ says that $\mcA(\alpha)$ is separable. Thus, let $B\subset\mcA(\alpha)\cap B(G)$ be a countable dense subset. By the previous lemma, for each $f\in B$ there is a representation $(\iota_f,\pi_f)$ of $(X_f,G)$ on a Hilbert space $\mcH_f$. We consider $$\mcH=\bigoplus_{f\in B}\mcH_f$$ and let $\pi=\bigoplus_{f\in B}\pi_f \colon G\to\UH$ be the orthogonal sum of the representations $\pi_f$. For each $f\in B$, let $w_f=\iota_f\alpha(1)\in\mcH_f$. Since $B$ is countable, by rescaling we may assume that $w=(w_f)_{f\in B}$ is summable, i.e., $w\in\mcH$. Now we define $\alpha'\colon G\to\mcH$ by $\alpha'(g)=\pi(g)w$, and let $Z$ be the weak closure of $\alpha'(G)$ in $\mcH$. Then the restriction $\alpha'\colon G\to Z$ is a Hilbert-representable compactification of~$G$, which we claim is isomorphic to~$\alpha$. To see this, we note first that $Z$ is $G$-isomorphic to a subspace of the product $\prod_{f\in B}X_f$; indeed, $\mcH$ (with the weak topology) is a subspace of the product space $\prod_{f\in B}\mcH_f$ (each $\mcH_f$ carrying the weak topology), hence the $G$-embedding $\prod_{f\in B}X_f\to\prod_{f\in B}\mcH_f$ (induced by the maps~$\iota_f$) restricts to a $G$-isomorphism between a subspace of $\prod_{f\in B}X_f$ and $Z$. Under this identification, the projection maps $Z\to X_f$ separate points of $Z$, and each composition $G\to Z\to X_f$ is just the compactification $G\to X_f$. Hence, by Stone--Weierstrass, $\mcA(\alpha')$ is the closed algebra generated by the algebras $\mcA(X_f)$, $f\in B$. Since $B$ is dense in $\mcA(\alpha)$, we deduce that $\mcA(\alpha')=\mcA(\alpha)$, which proves our claim.
\end{proof}

\begin{rem}\label{rem:Equiv Hilb-rep}\ 
\begin{enumerate}
\item\label{rem:Equiv Hilb-rep 1} A basic consequence of the first implication of the above proposition (which does not use the metrizability assumption) is that all Hilbert-representable compactifications of $G$ are factors of $H(G)$.
\item\label{rem:Equiv Hilb-rep 2} Another consequence is that $\Hilb(G)$ consists precisely of the $f\in\CG$ that factor through Hilbert-representable compactifications of $G$. One of the implications is immediate. For the other, if $f\in\Hilb(G)$, there is a sequence $f_n\to f$ with $f_n\in B(G)$, and hence $f$ factors through the compactification associated to the closed algebra $A$ generated by the functions $f_n$, which is metrizable (since $A$ is separable) and Hilbert-representable (by the proposition).
\end{enumerate}
\end{rem}

\begin{question}[\cite{glamegSurvey}, Question 5.12.3; \cite{megSelected}, Question 7.6]
  \label{question factors of H(G)}
  Are Hilbert-representable dynamical systems closed under factors? For ambits, and assuming $H(G)$ is metrizable, we may ask equivalently: are all factors of $H(G)$ Hilbert-representable?
\end{question}
This question has also been investigated in \cite{glaweiHilbert}. In Section~\ref{section factors of H(G)}, we will see that the answer is positive for pro-oligomorphic groups.

We should note that \emph{reflexively representable} dynamical systems are preserved under factors. In fact, the reflexively representable (or rather, when $W(G)$ is not metrizable, \emph{reflexively approximable}) compactifications of $G$ are exactly the factors of $W(G)$. See \cite{megHilb} and the references therein.

\noindent\hrulefill

\section{Semitopological semigroup compactifications}\label{Semitop semigroup comp}

\subsection{Definitions} A \emph{semitopological semigroup} is a semigroup that carries a topological structure such that the product operation is separately continuous (i.e., multiplying by an arbitrary fixed element on the left is continuous, and similarly on the right). We shall be interested in semitopological semigroups arising in the following manner.

\begin{defin}
A compactification $\alpha \colon G\to S$ is a \emph{semitopological semigroup compactification} if $S$ admits a semitopological semigroup law that makes $\alpha$ a homomorphism.
\end{defin}

\begin{rem}\label{rem:semitop semigp} Suppose $\alpha \colon G\to S$ is a semitopological semigroup compactification.
\begin{enumerate}
\item Then $S$ is in fact a monoid: $\alpha(1)$ is an identity.
\item By Lawson's joint continuity theorem, $\alpha(G)$ is a topological group (\cite{lawsonJoint}, Corollary~6.3).
\item\label{rem:semitop semigp.bi-inv} The compactification is bi-invariant.
\end{enumerate}
\end{rem}

Both the Hilbert and the WAP compactifications, $H(G)$ and $W(G)$, are semitopological semigroup compactifications. The semigroup law in $W(G)$ is given as follows: if $p=\lim g_i$ and $q=\lim h_j$, where $g_i,h_j$ belong to the homomorphic copy of $G$ in $W(G)$, then the product $pq$ is defined as the iterated limit $\lim_i\lim_j g_ih_j$. Grothendieck's double limit criterion (cf. Definition~\ref{def:wap}) and the bi-invariance of $W(G)$ ensure that the product is well-defined and yields a semitopological semigroup. The same construction works for $H(G)$ and other bi-invariant factors of $W(G)$. Conversely, one can use the double limit criterion to see that $W(G)$ is universal among semitopological semigroup compactifications of $G$ in the following sense.

\begin{fact}\label{factors of WAP} Let $S$ be a compactification of $G$. The following are equivalent:
\begin{enumerate}
\item $S$ is a semitopological semigroup compactification.
\item $S$ is a bi-invariant factor of $W(G)$.
\end{enumerate}
\end{fact}
\begin{proof}
See, for instance, \cite[Ch.\ III, \textsection 8]{bergCompactright}, Corollary 8.5.
\end{proof}

Given a reflexive Banach space $V$, the semigroup $\Theta(V)$ of linear contractions of $V$, $$\Theta(V)=\{T\in L(V):\|T\|\leq 1\},$$ endowed with the weak operator topology, is compact and semitopological. It turns out that every compact semitopological semigroup can be seen as a closed subsemigroup of $\Theta(V)$ for some reflexive Banach space $V$ \cite{shtern94,megreOperator}. Thus every compact semitopological semigroup is \emph{reflexively representable}.

\begin{defin}\label{Hilb-rep semigroups}
A semitopological semigroup $S$ is \emph{Hilbert-representable} if it can be embedded in the compact semitopological semigroup $\ThH$ of linear contractions of a Hilbert space $\mcH$.
\end{defin}

It is not difficult to see the following.

\begin{fact}\label{fact:universality of H(G)} Let $G$ be a topological group.
\begin{enumerate}
\item $H(G)$ is a Hilbert-representable semitopological semigroup, and is universal with this property among $G$-ambits (i.e., any other is a factor of $H(G)$).
\item In particular, $G$ is Eberlein if and only if $W(G)$ is Hilbert-representable as a semitopological semigroup.
\end{enumerate}
\end{fact}

The universality of $H(G)$ is clear (as per Remark~\ref{rem:Equiv Hilb-rep}(\ref{rem:Equiv Hilb-rep 1})) if we admit that the two notions of representability on Hilbert spaces discussed so far coincide on semigroup ambits, which is essentially the case:

\begin{fact}
Let $\alpha\colon G\to S$ be a metrizable semitopological semigroup compactification of $G$. Then, $S$ is a Hilbert-representable semitopological semigroup if and only if $\alpha$ is a Hilbert-representable compactification.
\end{fact}

See Lemma 4.5 in \cite{megHilb}. In the non-metrizable case, Definition \ref{Hilb-rep semigroups} is the correct property to consider, while Hilbert-representability of dynamical systems has to be relaxed. However, the semigroups that we study in this paper are metrizable.

\begin{defin}
Let $\alpha\colon G\to S$ be a semitopological semigroup compactification. We will say that $\alpha$ is \emph{$*$-closed} or, equivalently, that $\alpha$ is a \emph{semitopological $*$-semigroup compactification}, if the inverse operation on the group $\alpha(G)$ extends to a continuous map $^* \colon S\to S$. (Then $^*$ is automatically an \emph{involution on $S$}, i.e., $(p^*)^*=p$ and $(pq)^*=q^*p^*$ for every $p,q\in S$.)
\end{defin}

\begin{fact}\label{remark *-closed} The map $\alpha\colon G\to S$ is $*$-closed if and only if, whenever $f\in\mcA(\alpha)$, the function $g\mapsto f(g^{-1})$ is also in $\mcA(\alpha)$.
\end{fact}
\begin{proof}
  Let us denote the function $g\mapsto f(g^{-1})$ by $f^*$. Suppose $\alpha$ is $*$-closed and let $f\in\mcA(\alpha) = C(S)$. Then the function $p\mapsto f(p^*)$ belongs to $C(S)$, whence its restriction to $G$, which is $f^*$, belongs to $\mcA(\alpha)$. For the other direction, the 
involution is given by $p^*(f)=p(f^*)$ for $f\in\mcA(\alpha)$ and $p\in 
S=\widehat{\mcA(\alpha)}$.
\end{proof}

It follows readily that $W(G)$ is $*$-closed, for instance by applying Grothendieck's double limit criterion to check the above condition. By looking at unitary matrix coefficients, it is also easy to deduce that $H(G)$ is $*$-closed; more generally, we have the following.

\begin{prop}\label{Hilb rep is *-closed} Every Hilbert-representable semitopological semigroup compactification is $*$-closed.\end{prop}
\begin{proof}
Let $\alpha \colon G\to S$ be a compactification with an embedding $\beta \colon S\to\ThH$. It suffices to see that the image of $\beta$ is closed under the adjoint operation $^* \colon \ThH\to\ThH$; indeed, then we can define $s^*$ as the preimage of $\beta(s)^*$, and this gives a continuous map $^* \colon S\to S$ that extends the inverse operation on $\alpha(G)$. Now, if $s\in S$ is the limit of a net $\alpha(g_i)\in\alpha(G)$, then $\beta\alpha(g_i^{-1})$ converges to $\beta(s)^*$; by compactness, we may assume that $\alpha(g_i^{-1})$ converges to some $s'\in S$, so $\beta(s')=\beta(s)^*$. Hence $\beta(s)^*\in\beta(S)$.
\end{proof}

\subsection{Inverse semigroups}

In this short subsection, we review some general notions of the theory of semigroups, and some particular facts that hold for compact semitopological $*$-semigroups with a dense subgroup.

An element $e$ in a semigroup $S$ is an \emph{idempotent} if $e^2=e$. If $S$ has an involution~$^*$, then $e\in S$ is \emph{self-adjoint} if $e^*=e$.

\begin{defin} Let $S$ be a semigroup.
\begin{enumerate}[(i)]
\item An element $p\in S$ is \emph{regular} if there exists $q\in S$ such that $p=pqp$.
\item $S$ is \emph{regular} if every element is regular.
\item An element $q\in S$ is an \emph{inverse} for $p\in S$ if $p=pqp$ and $q=qpq$.
\item $S$ is an \emph{inverse semigroup} if every element has a unique inverse.
\end{enumerate}
\end{defin}

The canonical example of an inverse semigroup is the \emph{symmetric inverse semigroup} of all partial bijections of a set, with composition where it is defined.

A proof of the following general characterization can be found in \cite{howieFundamentals}, Theorem 5.1.1.

\begin{fact}\label{equivalence for inverse semigroups}
The following are equivalent for a semigroup $S$:
\begin{enumerate}
\item $S$ is an inverse semigroup.
\item $S$ is regular and the idempotents commute.
\end{enumerate}
\end{fact}

When a compact semitopological structure is available, and the semigroup contains a dense subgroup, much more is true. We formulate these additional properties in the case that we are interested in.

\begin{fact}\label{fact:Sp=Sqp etc}
Let $G\to S$ be a semitopological $*$-semigroup compactification.
\begin{enumerate} 
\item\label{Sp=Sqp} For every $p,q\in S$ we have $Sq=Spq$ if and only if $q=p^*pq$.
\item\label{idemp self-adj} Every idempotent is self-adjoint.
\item Let $e,f\in S$ be idempotents. The following are equivalent:
\begin{enumerate}
\item $e$ and $f$ commute.
\item $ef$ is also an idempotent.
\end{enumerate}
\item Let $p\in S$. The following are equivalent:
\begin{enumerate}
\item $p$ is regular.
\item $pp^*p=p$.
\item $p$ has a unique inverse.
\end{enumerate}
\item In particular, the following are equivalent:
\begin{enumerate}
\item $S$ in an inverse semigroup.
\item $S$ is regular.
\end{enumerate}
\end{enumerate}
\end{fact}
\begin{proof}\ 
\begin{enumerate}
\item See Lawson \cite{lawsonPoints}, Proposition 4.1.
\item This follows easily from (\ref{Sp=Sqp}).
\item One implication is immediate and the other is clear using that idempotents are self-adjoint.
\item Suppose $p$ is regular, so $p=pqp$ for some $q$. Hence $Spqp\subset Sqp\subset Sp=Spqp$, so $Sqp=Spqp$. Then, by (\ref{Sp=Sqp}), we have $qp=p^*pqp=p^*p$, and thus $p=pqp=pp^*p$.

\noindent If now we suppose that $p=pp^*p$, then $p^*=p^*pp^*$, so $p$ and $p^*$ are inverses. If $q$ is another inverse of $p$, then as before, we have $qp=p^*p$. Also, $q^*$ is an inverse for $p^*$, so $q^*p^*=pp^*$ and $pq=pp^*$. Then $q=qpq=qpp^*=p^*pp^*=p^*$.
\item Clear. \qedhere
\end{enumerate}
\end{proof}

\subsection{The WAP compactification of pro-oligomorphic groups}\label{The WAP compactification...}

In this subsection we will recall the model-theoretic description of the WAP compactification given in \cite{bentsa} for Roelcke precompact Polish groups. Since the results of the present paper are concerned with pro-oligomorphic groups, our presentation here will be restricted to these, i.e., to automorphism groups of \emph{classical} $\aleph_0$-categorical structures (as opposed to \emph{metric}). Still, it will be convenient to consider formulas as real-valued functions, taking values in $\{0,1\}$.

We refer to \cite{tenzie} for the necessary background in model theory and for the basics of $\aleph_0$-categorical structures. Let us recall the definition of the family of groups we will study.

\begin{defin} A group $G$ is \emph{oligomorphic} if it can be presented as a closed permutation group $G\leq S(X)$ of a countable set $X$ such that the orbit spaces $X^n/G$ are finite for every $n<\omega$; or, equivalently, if $G$ is the automorphism group of an $\aleph_0$-categorical, one-sorted structure.

A Polish group $G$ obtained as an inverse limit of oligomorphic groups will be called \emph{pro-oligomorphic}. Equivalently, $G$ is pro-oligomorphic if it can be presented as the automorphism group of an $\aleph_0$-categorical, multi-sorted structure. These are exactly the Roelcke precompact, non-archimedean, Polish groups; see \cite{tsaUnitary}, Theorem 2.4.
\end{defin}

Throughout this paper, whenever $G$ is a pro-oligomorphic group and we write $G=\Aut(M)$, we understand that $M$ is an $\aleph_0$-categorical structure and $G$ is its automorphism group. By the homogeneity of $\aleph_0$-categorical structures, we have the following.

\begin{fact}\label{E(M)} Let $G = \Aut(M)$ be a pro-oligomorphic group and $\widehat{G}_L$ be the completion of $G$ with respect to its left uniformity, which is a topological semigroup. Then $\widehat{G}_L$ can be identified with the topological semigroup $E(M)$ of elementary embeddings of $M$ into itself with the topology of pointwise convergence.\end{fact}
\begin{proof}
Let $\xi\in M^\omega$ be an enumeration of $M$ and define the distance $d_L$ on $E(M)$ by $d_L(x,y)=\sup_{i<\omega}2^{-i}d(x(\xi_i),y(\xi_i))$, where $d$ is the discrete, $\{0, 1\}$-valued distance on $M$. It induces the topology of pointwise convergence on $E(M)$. The restriction of $d_L$ to $G$ is a compatible, left-invariant metric, which induces the left uniformity of $G$. By homogeneity, $G$ is dense in $E(M)$ with respect to $d_L$. Since, moreover, $E(M)$ is complete with respect to $d_L$, it is the left completion of $G$.
\end{proof}

Recall that if $(X, d)$ is a metric space and $G$ acts on $X$ by isometries, then
\begin{equation*}
  X \sslash G = \{\cl{Gx} : x \in X \}
\end{equation*}
is a metric space with induced distance
\begin{equation} \label{eq:distance-quotient}
  d(\cl{Gx}, \cl{Gy}) = \inf \{d(x, gy) : g \in G \}.
\end{equation}
When $X$ is complete, so is $X \sslash G$.

One important instance of this construction is
\begin{gather*}
  R(G)=(\widehat{G}_L\times\widehat{G}_L)\sslash G,
\end{gather*}
where $G$ acts diagonally on $\widehat{G}_L\times\widehat{G}_L$ by left translation.
Given elements $x,y\in \widehat{G}_L$, we denote the class of $(x,y)$ in $R(G)$ by $[x,y]_R$.
The group $G$ embeds densely in $R(G)$ via the map $g \mapsto [1,g]_R = [g^{-1},1]_R$: if $g_n\to x$ and $h_n\to y$ in $\widehat{G}_L$, we will have $[1,g_n^{-1}h_n]_R\to [x,y]_R$ in $R(G)$. This makes $R(G)$ a completion of $G$ with respect to the distance
\begin{gather*}
  d(g,f) = \inf \, \bigl\{ d_L(1,h) + d_L(g,hf) : h \in G \bigr\}
\end{gather*}
coming from \eqref{eq:distance-quotient}. Two group elements $g$ and $f$ are close in $R(G)$ if and only if there is $h \in G$ such that $d_L(1,h) + d_L(g,hf)$ is small.
Letting $h' = g^{-1}hf$, we see that $g$ and $f$ are close in $R(G)$ if and only if there exist $h,h'$ close to $1$ such that $f = h^{-1}gh'$.
In other words, the distance on $R(G)$ induces on $G$ the \emph{Roelcke uniformity}, namely the infimum of the left and right uniformities, and the completion $R(G)$ is the \emph{Roelcke completion} of $G$.
The group $G$ is Roelcke precompact precisely when $R(G)$ is compact.
That is, when the completion $R(G)$ coincides with the compactification of $G$ associated to the algebra $\UC(G)$.
Since the completion $R(G)$ is metrizable by construction, for Roelcke precompact Polish groups, this is a metrizable compactification, and so are all its factors.

For the rest of this section, we fix a pro-oligomorphic group $G=\Aut(M)$.
By Fact~\ref{E(M)}, in this case we can write $R(G)=(E(M)\times E(M))\sslash G$, allowing us to identify formulas with Roelcke uniformly continuous functions.
Indeed, let $\varphi(u,v)$ be a formula and $a,b \in M$ tuples of the appropriate length.
The function $(x,y) \mapsto \varphi\bigl( x(a), y(b) \bigr)$ is continuous on $E(M)^2$ and $G$-invariant, so it factors via $R(G)$:
\begin{gather*}
  \varphi_{a,b} \bigl( [x,y]_R \bigr) = \varphi\bigl( x(a), y(b) \bigr).
\end{gather*}
Its restriction to $G \subset R(G)$, namely $g \mapsto \varphi(a,gb)$, is therefore in $\UC(G)$.
Conversely, by \cite{bentsa}, Theorem~5.4, such functions generate a dense subalgebra of $\UC(G)$.
Therefore, the functions $\varphi_{a,b}$ separate points of $R(G)$ or, in other words, $[x,y]_R\in R(G)$ is determined by the values $\varphi(x(a),y(b))$, where $\varphi(u,v)$ varies over the formulas of $M$ and $a,b$ vary over tuples of $M$ of the appropriate length. Coding $x\in E(M)$ by a tuple $\tilde{x} = x(\xi) \in M^\omega$, where $\xi\in M^\omega$ is a fixed enumeration of $M$, we see that $[x,y]_R$ can be identified with the type $\tp(\tilde{x},\tilde{y})$.

By Gelfand duality, factors of $R(G)$ correspond to closed subalgebras of $\UC(G)$: of these, we will mostly concentrate on $\UC(G) \supset \WAP(G) \supset \Hilb(G)$.
Interestingly, the correspondence between $\UC(G)$ and formulas gives rise to correspondences between these subalgebras and special classes of formulas that have been independently studied in model theory.

For the subalgebra $\WAP(G)$, this correspondence was treated in \cite{bentsa}.
Its Gelfand space is the WAP compactification $W(G)$, which is therefore a factor of $R(G)$.
We will denote the image of $[x,y]_R$ in $W(G)$ simply by $[x,y]$: it is determined by the values of WAP functions at $[x,y]_R$.
A formula $\varphi(u,v)$ is stable (when restricted to $\tp(a,b)$, which is a definable set by $\aleph_0$-categoricity) if and only if $\varphi_{a,b}(g) = \varphi(a,gb)$ is WAP, and conversely, such functions generate a dense subalgebra of $\WAP(G)$ (\cite{bentsa}, Theorem~5.4).
Therefore, $[x,y]$ is determined by the values $\varphi(x(a),y(b))$, as before, only that $\varphi(u,v)$ ranges over the stable formulas.
In particular, $G$ is a WAP group if and only if the theory of $M$ is stable.

The canonical $G$-map $G\to W(G)$ is given by
$$g\mapsto [1,g],$$
and the $G$-action by
$$g[x,y]=[xg^{-1},y].$$
The involution $^* \colon W(G)\to W(G)$ extending the inverse on the image of $G$ is given by $$[x,y]^*=[y,x].$$
Moreover, the semitopological semigroup law of $W(G)$ can be described in terms of the \emph{stable independence relation} of $M$. In order to explain this, we first recall the definition of \emph{imaginaries} and some notions from stability theory.

Let $M$ be a structure. An \emph{imaginary element} of $M$ is the a class of a definable equivalence relation on some finite power of $M$. In other words, if a formula $\varphi(u,v)$ defines an equivalence relation on $M^n$, then each class $[a]_\varphi\in M^n/\varphi$ is an imaginary of $M$.

A standard model-theoretic construction allows us to consider all the imaginaries of $M$ as actual elements in a larger (multi-sorted) structure, denoted $M^\eq$. See \cite[\textsection 8.4]{tenzie} for the details. This enlargement of $M$ is in many senses innocuous; in particular, the natural restriction map $\Aut(M^\eq)\to\Aut(M)$ is an isomorphism between their automorphism groups. Thus, for many purposes, it is convenient to work directly with the structure $M^\eq$.

Moreover, imaginary elements of $\aleph_0$-categorical structures are in correspondence with the open subgroups of its automorphism group. Indeed, a subgroup $V\leq G$ is open if and only if it is the stabilizer of an imaginary element of $M$. That is to say, if and only if there is a definable equivalence relation $\varphi(u,v)$ and a tuple $c$ such that $$V=\{g\in G: M \models \varphi(c,gc)\}=\{g\in G:[c]_\varphi=g[c]_\varphi\}.$$ See, for example, \cite[\textsection 5]{tsaUnitary}.

A special kind of imaginary is given as follows. If $\varphi(u,v)$ is any formula, we can define a formula $E_\varphi(u,u')$ by $$E_\varphi(u,u'):=\forall v(\varphi(u,v)\leftrightarrow\varphi(u',v)).$$ Then $E_\varphi$ defines an equivalence relation on $M^{|u|}$. An imaginary $[c]_{E_\varphi}\in M^{|u|}/E_\varphi$ should be seen as representing the formula $\varphi(c,v)$; $[c]_{E_\varphi}$ (also denoted simply by $[c]_\varphi$) is called the \emph{canonical parameter} of $\varphi(c,v)$,

Given a type $t\in S_u(M)$ and a formula $\varphi(u,v)$, the \emph{$\varphi$-definition} of $t$ is the function $d_t\varphi \colon M^{|v|}\to\{0,1\}$ given by $$d_t\varphi(b):=\varphi(u,b)^t,$$
where the right term denotes the value of $\varphi(u,b)$ in the type $t$. Then, the formula $\varphi(u,v)$ is \emph{stable} if and only if $d_t\varphi$ is an $M$-definable predicate for every $t\in S_u(M)$. (We only need to consider the model $M$ because of $\aleph_0$-categoricity; for a treatment of stability in the general case, see for instance \cite{tenzie}, Ch.~8.) In this case we can write $d_t\varphi(v)$ in the form $\psi(c,v)$, and then consider the canonical parameter of this formula; we denote this canonical parameter by $\Cb_\varphi(t)$. (The choice of the formula $\psi$ can be done uniformly in $t$, that is, $c$ depends on $t$ but $\psi(w,v)$ does not.) The tuple $$\Cb(t)=(\Cb_\varphi(t))_{\varphi\text{ stable}}$$ is the \emph{canonical base} of $t$.

Finally, an element $d\in M^\eq$ is in the \emph{algebraic closure} of a set $A\subset M^\eq$ if, for some finite tuple $a\subset A$, $d$ has only finitely many conjugates by automorphisms fixing $a$. We denote the algebraic closure of $A$ by $\acl(A)$ (which is always a subset of $M^\eq$). The set $A$ is \emph{algebraically closed} if $A=\acl(A)$.

\begin{fact}
Let $a\in (M^\eq)^{|u|}$ be a tuple and $B\subset M^\eq$ be any subset. There is an extension of the type $\tp(a/\acl(B))$ to a type $t\in S_u(M)$ such that $\Cb(t)\subset\acl(B)$. Moreover, $\Cb(t)$ does not depend on the particular extension; in other words, if $s\in S_u(M)$ is another such extension, then $d_t\varphi=d_s\varphi$ for every stable formula $\varphi(u,v)$.
\end{fact}

\begin{defin}\ 
\begin{enumerate}[(i)]
\item If $a$, $B$ and $t$ are as in the previous fact, we define $\Cb_\varphi(a/B):=\Cb_\varphi(t)$, $\Cb(a/B):=\Cb(t)$.
\item Given any sets $A,B,C\subset M^\eq$, we say that \emph{$A$ is stably independent from $C$ over $B$}, denoted $$A\ind_BC,$$ if for any tuple $a\in A^{|u|}$ we have $\Cb(a/B)=\Cb(a/BC)$.
\item If $a,c$ are tuples from $M^\eq$ and $B$ is any subset, we write $a\equiv^s_Bc$ to mean that $a$ and $c$ have the same \emph{stable type} over $B$, that is, $\varphi(a,b)=\varphi(c,b)$ for any stable formula $\varphi(u,v)$ and parameter $b\in B^{|v|}$. When $B$ is empty we shall write simply $a\equiv c$, since $a\equiv^s_\emptyset c$ is indeed equivalent to $\tp(a/\emptyset)=\tp(c/\emptyset)$.
\end{enumerate}
\end{defin}

Note that the natural identification of $\Aut(M)$ and $\Aut(M^\eq)$ extends to an identification of $E(M)$ and $E(M^\eq)$.

\begin{convention}
We may consider the elements of $E(M)$ as sets (notably, to apply the relations $\ind$ and $\equiv^s$ to them), and this shall be done in the following way: an element $x\in E(M)$ is interpreted as the set $x(M^\eq)\subset M^\eq$ (which is the same as $\acl(x(M))$). For instance, $x\cap y$ will denote $x(M^\eq)\cap y(M^\eq)$.

If appearing as arguments of the relation $\equiv^s$, the elements of $E(M)$ will be considered as infinite tuples indexed by $M$ (or by $\omega$ via a fixed enumeration $\xi$, as before).

In these contexts, the juxtaposition $xy$ will denote the juxtaposed tuple (or merely the union of sets).
\end{convention}

The pair $(\ind,\equiv^s)$ satisfies the following usual properties.

\begin{fact}\label{properties of ind,equiv} Let $x,y,z,w$ be any tuples from $M^\eq$.
\begin{enumerate}
\item (Invariance) If $x\ind_yz$ and $xyz\equiv x'y'z'$, then $x'\ind_{y'}z'$. If $x\ind_yz$ and $x\equiv^s_{yz}x'$, then $x'\ind_yz$.
\item (Symmetry) If $y\subset x\cap z$, then $x\ind_yz$ if and only if $z\ind_yx$.
\item (Transitivity) $x\ind_yzw$ if and only if $x\ind_{yz}w$ and $x\ind_yz$.
\item (Existence) There exist $x',y',z'$ such that $x'y'\equiv xy$, $y'z'\equiv yz$ and $x'\ind_{y'}z'$.
\item (Stationarity) Suppose $y$ is algebraically closed. If $x\equiv^s_yz$, $x\ind_yw$ and $z\ind_yw$, then $x\equiv^s_{yw}z$.
\item (Non-triviality) If $x\ind_yz$, then $\acl(x)\cap\acl(z)\subset\acl(y)$.
\end{enumerate}
\end{fact}
\begin{proof}
We refer the reader to \cite[Ch.\ 1, \textsection 2]{pil96}.
\end{proof}

\begin{fact}
The semigroup law in $W(G)$ is given by $$[x,y][y,z]=[x,z]\text{ \ if \ }x\ind_yz.$$ The properties of the independence relation stated above ensure that, for any $p,q\in W(G)$, we can always find $x,y,z\in E(M)$ such that $p=[x,y]$, $q=[y,z]$ and $x\ind_yz$.
\end{fact}

The latter allows for a model-theoretic description of the idempotents of $W(G)$. This was given in \cite[\textsection 5]{bentsa}. Let us end this section by recalling this description and giving a complete proof. Moreover, we complement it with a characterization of the regular elements of $W(G)$, which will be used in our main result.

For the definition and properties of the $\varphi$-rank see \cite[Ch.\ 1, \textsection 3]{pil96}.

\begin{lem}\label{idempotents and regular elements} Let $p=[x,y]\in W(G)$, $C=x\cap y$.
\begin{enumerate}
\item The following are equivalent:
\begin{enumerate}
\item $p$ is an idempotent (i.e., $pp=p$).
\item $x\equiv^s_Cy$ and $x\ind_Cy$.
\end{enumerate}
\item The following are equivalent:
\begin{enumerate}
\item $p$ is regular (i.e., $pp^*p=p$).
\item $x\ind_Cy$.
\end{enumerate}
\end{enumerate}
\end{lem}
\begin{proof}
(1) By replacing $x,y$ by an equivalent pair if necessary, we can find $z\in E(M)$ such that $x\ind_yz$ and $xy\equiv yz$. Indeed, using the saturation of $\aleph_0$-categorical structures to replace tuples by appropriate equivalent tuples if necessary, we may assume first that there is $x'$ with $xy\equiv yx'$. Then, by \emph{existence}, we may assume there is $z$ with $yx'\equiv yz$ and $x\ind_y z$. Thus $z$ satisfies the conditions above. In particular, $pp=[x,y][y,z]=[x,z]$.

Suppose $p$ is an idempotent. Then $[x,y]=[x,z]$, i.e., $y\equiv^s_xz$. From this we see that $C=x\cap z\subset y\cap z$. Actually, we can deduce that $x\equiv^s_Cy$ and $C=y\cap z$. Indeed, if $x(a)=y(b)\in C$, then from $y\equiv^s_xz$ we get $x(a)=z(b)$ and thus $y(b)=z(b)$. Since $xy\equiv yz$, then $x(b)=y(b)$. Hence $a=b$. If we denote $D=x^{-1}(C)$, we see that the restrictions $x|_D=y|_D$ coincide. Thus, with the appropriate orderings, $xC\equiv MD\equiv yC$, so in particular $x\equiv^s_Cy$. Furthermore, since $x\cap y=x(D)$ and $xy\equiv yz$, we have $y\cap z=y(D)=C$.

Next we argue that $x\ind_zy$. This is equivalent to show that, for every stable formula $\varphi$, the $\varphi$-rank of $x$ over $yz$ equals the $\varphi$-rank of $x$ over $z$: $R_\varphi(x/yz)=R_\varphi(x/z)$. But indeed, since $x\ind_yz$ and $[x,y]=[x,z]$, we have $$R_\varphi(x/yz)=R_\varphi(x/y)=R_\varphi(x/z).$$ To see why the second identity holds, take $x'$ such that $xz\equiv x'y$. Then $[x,y]=[x',y]$, so $x$ and $x'$ have the same $\varphi$-type over $y$, and thus $R_\varphi(x/y)=R_\varphi(x'/y)=R_\varphi(x/z)$.

Now we have $x\ind_yz$ and $x\ind_zy$, and thus $\Cb(x/yz)\subset y\cap z=C$. This implies that $x\ind_Cy$.

Conversely, suppose $x\equiv^s_Cy$ and $x\ind_Cy$. Together with the conditions $xy\equiv yz$, $x\ind_yz$, this implies that $C=y\cap z$, $y\equiv^s_Cz$ and, by \emph{transitivity}, $x\ind_Cz$. Then, by \emph{symmetry} and \emph{stationarity}, we have $y\equiv^s_xz$, i.e., $[x,y]=[x,z]$, which means that $p$ is an idempotent.

(2) Replacing $x,y$ by an equivalent pair if necessary, we can find $z,w\in E(M)$ with $x\ind_yz$, $xy\equiv zy$, $x\ind_zw$ and $xy\equiv zw$. In particular, $$pp^*p=[x,y][y,z][z,w]=[x,w].$$ In addition, from $xy\equiv zy$ we get $C\subset z$, and from $x\ind_yz$, by \emph{non-triviality}, we get $x\cap z\subset y$. Hence, $C=x\cap z$.

Suppose $p$ is regular. The condition $p=pp^*p$ becomes $[x,y]=[x,w],$ so $y\equiv^s_xw$. Moreover, since $e=pp^*=[x,z]$ is an idempotent, we have that $x\ind_Cz$. From this and $x\ind_zw$ we obtain $x\ind_Cw$. Since $y\equiv^s_xw$, we deduce that $x\ind_Cy$.

Suppose conversely that $x\ind_Cy$. From our hypothesis we get $x\ind_Cz$, and then $x\ind_Cw$. The condition $xy\equiv zy$ implies $x\equiv^s_Cz$, and this together with $xy\equiv zw$ implies $y\equiv^s_Cw$. Hence, by stationarity, $y\equiv^s_xw$. That is, $pp^*p=[x,w]=[x,y]=p$.
\end{proof}

\noindent\hrulefill

\section{The Fourier--Stieltjes algebra of pro-oligomorphic groups}\label{B(G) of olig groups}

\subsection{Examples of functions in $\Hilb(G)\setminus B(G)$} As mentioned before, the Fourier--Stieltjes algebra $B(G)$ is, as a general rule, strictly contained in its closure $\Hilb(G)$. For example, if $G$ is compact, then $B(G)$ is not closed in $\CG$ unless $G$ is finite (see for instance \cite{hewittAbstractII}, Theorem~37.4). Let us begin this section with a model-theoretic argument showing that the same holds for pro-oligomorphic groups.

For locally compact groups, the algebra $C_0(G)$ of functions vanishing at infinity is always contained in $\Hilb(G)$. We recall that a function $f\in C(X)$ on a locally compact space $X$ \emph{vanishes at infinity} if for every $\epsilon>0$ there is a compact set $K\subset X$ such that $|f(x)|<\epsilon$ for every $x$ outside~$K$. These functions can be extended continuously to the one-point compactification $X'=X\cup\{\infty\}$ of $X$ by setting $f(\infty)=0$. In the case of a locally compact group $G$ (say, with Haar measure~$\mu$), the one-point compactification of $G$ is a Hilbert-representable semitopological semigroup: it can be embedded into the linear contractions of $L^2(G,\mu)$ by sending $\infty$ to the zero operator, and otherwise extending the regular representation of $G$. Thus, $C_0(G)\subset\Hilb(G)$.

Similarly, for closed subgroups of $S_\infty$, we have a simple way to produce functions in $\Hilb(G)$. Recall that if a group $G$ acts continuously on a discrete set $X$, then we have a natural unitary representation $\pi \colon G\to U(\ell^2(X))$ defined (on the canonical basis of $\ell^2(X)$) by $\pi(g)e_x=e_{gx}$.

\begin{lem}\label{vanishing functions are Hilb}
Let $M$ be a structure, $G=\Aut(M)$. Let $F \colon M^n\to\mbC$ be a function vanishing at infinity and let $a\in M^n$. Then the function $f\in\CG$ given by $f(g)=F(ga)$ belongs to $\Hilb(G)$.
\end{lem}
\begin{proof}
We can assume that $F$ is zero everywhere except on a finite set $B\subset M^n$, since the general case can be uniformly approximated by instances of this form. Take the natural representation $\pi \colon G\to U(\ell^2(M^n))$ and the vectors $v=\sum_{b\in B}F(b)e_b$, $w=e_a$. Then we have $f(g)=\langle v,\pi(g)w\rangle$, which shows that $f\in B(G)$.
\end{proof}

It is convenient to introduce the following definition. Given an action by isometries $G\actson X$ and a sequence $(x_i)_{i<\omega}\subset X$, let us say that $(x_i)$ is \emph{indiscernible} if for all indices $i_1 < i_2 < \cdots < i_k$ and $j_1 < j_2 < \cdots < j_k$ we have the equality $$[x_{i_1},x_{i_2},\dots,x_{i_k}]=[x_{j_1},x_{j_2},\dots,x_{j_k}]$$ in $X^k\sslash G$. We remark that for the natural action $G\actson M$ of the automorphism group of an $\aleph_0$-categorical structure (classical or metric), by (approximate) homogeneity, this definition coincides with the usual model-theoretic notion of an indiscernible sequence.

The following folklore lemma characterizes indiscernible sequences in Hilbert spaces.
\begin{lem}\label{lem:indisc in Hilbert}
Let $(w_i)_{i<\omega}$ be a sequence of vectors in a Hilbert space $\mcH$, and suppose $(w_i)$ is indiscernible for the action $\UH\actson\mcH$. Then, there are $w',w'_i\in\mcH$ such that $w'_i\perp w'$, $\|w'_i\|=\|w'_j\|$, $w'_i\perp w'_j$ for every $i\neq j$, and $w_i=w'+w'_i$ for every $i$. In particular, $w'$ is the weak limit of $(w_i)_{i<\omega}$.
\end{lem}
\begin{proof}
  Note that, by homogeneity of the Hilbert space, a sequence $(w_i)$ is indiscernible iff all $w_i$ have the same norm and $\langle w_i, w_j \rangle$ is constant for $i \neq j$. Let $w'$ be a weak accumulation point of the $w_i$. Then by indiscernibility, for all $i \neq j$,
  \begin{equation*}
    \langle w_i, w_j \rangle = \langle w', w_i \rangle = \langle w', w' \rangle.
  \end{equation*}
  Setting $w_i' = w_i - w'$, we easily obtain the claimed properties.
\end{proof}

In the following proposition we suppose $G$ is pro-oligomorphic, so in particular $E(M) = \widehat{G}_L$ as per Fact~\ref{E(M)}, and indiscernible sequences for the natural action $G\actson \widehat{G}_L$ are the same as indiscernible sequences in $E(M)$ in the usual model-theoretic sense. We note also that every function $f \in B(G)$, being left uniformly continuous, extends to a function on $E(M)$.

We prove that functions in $B(G)$ vanishing at infinity must decay at a certain rate along indiscernible sequences.

\begin{prop} \label{p:vanish-quickly}
Let $G$ be a pro-oligomorphic group, say $G=\Aut(M)$. Let $F \colon M^n\to\mbC$  be a function vanishing at infinity, $a\in M^n$, and let $f \colon E(M)\to\mbC$ be given by $f(x)=F(x(a))$.

Suppose $(x_i)_{i<\omega}\subset E(M)$ is an indiscernible sequence such that $(x_i(a))_{i<\omega}$ is non-constant. If $f|_G\in B(G)$, then $$\left|\frac{1}{n}\sum_{i<n}f(x_i)\right| = O(\frac{1}{\sqrt{n}}),$$
where the implicit constant depends only on $f$.
\end{prop}
\begin{proof}
Suppose we have a continuous unitary representation $\pi \colon G\to\UH$ such that $f(g)=\langle v,\pi(g)w\rangle$ for all $g\in G$. Being a homomorphism, $\pi$ is left uniformly continuous, so it extends to a representation $\pi \colon E(M)\to E(\mcH)$. (Here, $E(\mcH)$ is the semigroup of isometric linear endomorphisms of $\mcH$, which is also the left completion of $U(\mcH)$.) We have $f(x)=\langle v,\pi(x)w\rangle$ for all $x\in E(M)$.

Since $(x_i)\subset E(M)$ is indiscernible for the action of $G$, the sequence $(w_i)\subset\mcH$ given by $w_i=\pi(x_i)w$ is indiscernible in the Hilbert space $\mcH$ for the action of $\pi(G)$, and thus also for the action of $U(\mcH)$. Let $w'$ and $w'_i$ be as given by Lemma~\ref{lem:indisc in Hilbert}. Since $F$ vanishes at infinity and $(x_i(a))$ is indiscernible and non-constant, we have that $f(x_i)\to 0$.
That is, $\langle v,w'\rangle=0$. We deduce that $$\left|\frac{1}{n}\sum_{i<n}f(x_i)\right|=\left|\langle v,\frac{1}{n}\sum_{i<n}w'_n\rangle\right|\leq\frac{\|v\|\cdot\|\sum_{i<n}w_i'\|}{n}=\frac{\|v\|\sqrt{\sum_{i<n}\|w'_i\|^2}}{n}=\frac{\|v\|\cdot\|w'_0\|}{\sqrt{n}}\leq\frac{\|v\|\cdot\|w\|}{\sqrt{n}}$$
\end{proof}

\begin{cor}
Let $G$ be pro-oligomorphic and infinite. Then $B(G)$ is not closed in the uniform norm.
\end{cor}
\begin{proof}
Choose any non-constant indiscernible sequence $(x_i)\subset E(M)$ (which always exists if $M$ is $\aleph_0$-categorical) and an element $a\in M$ such that $(x_i(a))$ is non-constant. Then take $F \colon M\to\mbC$ vanishing at infinity and such that $F(x_i(a))=1/i^{1/3}$. Then, by Lemma~\ref{vanishing functions are Hilb} and Proposition~\ref{p:vanish-quickly}, we obtain that the function defined by $f(g)=F(ga)$ is in $\Hilb(G)$ but not in $B(G)$.
\end{proof}

\subsection{A model-theoretic description of the Hilbert compactification}\label{Description of H(G)}

As explained in Subsection~\ref{The WAP compactification...}, the WAP compactification of a pro-oligomorphic group $G$ is the space of types of pairs of embeddings $x,y\in E(M)$ \emph{restricted to stable formulas}. Dually, this can be stated by saying that $\WAP(G)$ is the closed algebra generated by the functions of the form $$\varphi_{a,b}(g) = \varphi(a,gb),$$ where $\varphi(u,v)$ is a stable formula and $a,b$ are tuples from $M$.
For a more detailed explanation of this duality see \cite[\textsection 5]{bentsa} or \cite[\textsection 4]{iba14}.
Hence it is natural to ask which formulas $\varphi(u,v)$ give rise to functions in the subalgebra $\Hilb(G)$.

We start with the following basic observation.

\begin{lem}
Let $M$ be a structure, $G=\Aut(M)$. Let $\varphi(u,v)$ be a formula defining an equivalence relation on $M^n$ and let $a,b\in M^n$. Then the function $\varphi_{a,b}$ (which takes the value 1 if the elements are related and 0 otherwise) is in $B(G)$.
\end{lem}
\begin{proof}
It suffices to consider the natural representation $\pi \colon G\to U(\ell^2(M^n/\varphi))$ and observe that $\varphi_{a,b}(g)=\langle e_{[a]_\varphi},\pi(g)e_{[b]_\varphi}\rangle$.
\end{proof}

The reader can also check that $\varphi_{a,b}$ belongs to $B(G)$ under the weaker assumption that $\varphi(x,b)$ defines a \emph{weakly normal set}, that is to say, that the canonical parameter of $\varphi(x,b)$ is in the algebraic closure of any tuple $a$ that satisfies the formula.

We want to give a converse to the previous lemma, for $\aleph_0$-categorical structures. For this we invoke the classification theorem of unitary representations of pro-oligomorphic groups proved in \cite{tsaUnitary}.

\begin{fact}[Classification Theorem]
Let $G$ be a pro-oligomorphic group.
\begin{enumerate}
\item Every continuous unitary representation of $G$ is a direct sum of irreducible representations.
\item Every irreducible continuous unitary representation is a subrepresentation of the quasi-regular representation $\pi_V \colon G\to U(\ell^2(G/V))$ for some open subgroup $V\leq G$.
\end{enumerate}
\end{fact}

\begin{prop}\label{Hilb by formulas}
Let $G$ be pro-oligomorphic, $G=\Aut(M)$. Then $\Hilb(G)$ is the closed linear span of the functions of the form $$\varphi_{a,b}(g) = \varphi(a,gb),$$ where $\varphi(u,v)$ is a definable equivalence relation on some power $M^n$ and $a,b$ are tuples in $M^n$.
\end{prop}
\begin{proof}
It suffices to show that every $f\in B(G)$ can be uniformly approximated by linear combinations of functions of this form. By the classification theorem, every continuous unitary representation is a subrepresentation of one of the form $\pi \colon G\to U\left(\bigoplus_k\ell^2(G/V_k)\right)$, where each $V_k$ is an open subgroup of~$G$. Now, every matrix coefficient of $\pi$ can be uniformly approximated by a linear combination of \emph{basic} matrix coefficients, that is, given by $$g\mapsto\langle e_{g_0V_k},\pi(g)e_{g_1V_k}\rangle$$ for vectors $e_{g_0V_k}$, $e_{g_1V_k}$ from the canonical basis of $\ell^2(G/V_k)$. Fix an open subgroup $V=V_k$; it is the stabilizer of some imaginary element $[c]_\varphi\in M^\eq$. If we take $a=g_0c$ and $b=g_1c$, we have that $g_0V=gg_1V$ if and only if $[a]_\varphi=g[b]_\varphi$. In other words, $$\langle e_{g_0V},\pi(g)e_{g_1V}\rangle=\varphi(a,gb).$$ The proposition follows.
\end{proof}

Dually, this characterization of $\Hilb(G)$ will provide a nice model-theoretic description of the Hilbert compactification $H(G)$.

We fix a pro-oligomorphic group $G=\Aut(M)$. Let $K=M^\eq\cup\{\infty\}$ be the one-point compactification of $M^\eq$, and let
\begin{equation*}
  \Xi = \{p \in K^K : p(\infty) = \infty \text{ and $p$ is injective on } p^{-1}(M^\eq) \}.
\end{equation*}
Then $\Xi$, equipped with composition and the product topology, is a compact semitopological inverse semigroup (in fact, isomorphic to the semigroup of partial bijections of $M^\eq$). Let $P(M)=\overline{G}\subset \Xi$ be the closure of $G$ in the product space $K^K$ (where we set $g(\infty)=\infty$ for every $g\in G$). Then, if we think of an element $p\in K^K$ as a partial map $M^\eq\to M^\eq$ (undefined on $a$ whenever $p(a)=\infty$), we get the following.

\begin{prop}
The elements of $P(M)$ are precisely the partial elementary maps of $M^\eq$ with algebraically closed domain. Besides, $P(M)$ is closed under composition, and with this operation it becomes a semitopological $*$-semigroup compactification of $G$, which is moreover an inverse semigroup.
\end{prop}
\begin{proof}
  It is clear that any $p\in P(M)$ is a partial elementary map of $M^\eq$, and also that its domain must be algebraically closed. Conversely, let $p \colon A\to M^\eq$ be an elementary map with $A$ algebraically closed. Fix a finite tuple $a$ from $A$, a finite tuple $b$ disjoint from $A$ and a finite subset $C\subset M^\eq$ (intended as the complement of a neighbourhood of $\infty$ in $K$); denote $a'=p(a)$. Choose a tuple $b'$ such that $ab\equiv a'b'$, then take $b''$ satisfying $b''a'\equiv b'a'$ and $b''\ind_{a'}C$. Since $b$ is disjoint from $\acl(a)\subset A$ we have that $b''$ is disjoint from $\acl(a')$, whence $b''$ is disjoint from $C$. Now, by homogeneity there is $g\in G$ such that $ga=a'$ and $gb=b''$. This shows that $p$ can be approximated by elements of $G$ in the topology of $K^K$.

Finally, $P(M)$, being a closed subsemigroup of $\Xi$ closed under inverses, is also a compact inverse semitopological semigroup.
\end{proof}

We remark that we have defined $P(M)$ directly as a family of partial maps on $M^\eq$, and not on~$M$. Unlike the case of $E(M)$, which can be identified with $E(M^\eq)$, the previous construction applied to $M$ would yield a smaller object (a factor of $P(M)$), which may lose information. However, we have the following.

\begin{rem}\label{rem:weak elim imag}
  The structure $M$ has \emph{weak elimination of imaginaries} (see for instance \cite[\textsection 8.4]{tenzie}) if and only if every algebraically closed set $A\subset M^\eq$ is equal to the definable closure of $A\cap M$. It follows that $M$ has weak elimination of imaginaries if and only if $P(M)$ coincides with its factor consisting of partial elementary maps of $M$ with (relatively) algebraically closed domain.
\end{rem}

We also observe that $P(M)$ can be alternatively defined as the closure of the image of $G$ inside $\Theta(\ell^2(M^\eq))$, induced by the natural unitary representation $G\to U(\ell^2(M^\eq))$. Indeed, by identifying $\infty\in K$ with the zero of the Hilbert space, we have natural topological embeddings
$$G \subset \Xi \subset\Theta(\ell^2(M^\eq)).$$
In particular, $P(M)$ is a factor of the Hilbert compactification.

\begin{theorem}\label{H(G)=P(M)}
Let $G=\Aut(M)$ be a pro-oligomorphic group. Then $P(M)$ coincides with the Hilbert compactification $H(G)$.
\end{theorem}
\begin{proof}
  This follows from the previous observation and the fact, implied by the classification theorem, that every separable continuous unitary representation of $G$ is a subrepresentation of $G\to U(\ell^2(M^\eq))$.
  Nevertheless, let us give an explicit isomorphism $H(G)\to P(M)$ based on the model-theoretic description of $W(G)$.
  Given endomorphisms $x,y\in E(M^\eq)$, let $[x,y]_H$ denote the image of $[x,y]\in W(G)$ under the canonical map $W(G)\to H(G)$: it is determined by the values of all $f \in \Hilb(G)$ at $[x,y]$ (or at $[x,y]_R$, for that matter).
  By Proposition~\ref{Hilb by formulas}, these values are in turn determined by the values $\varphi_{a,b}\bigl( [x,y] \bigr) = \varphi(x(a),y(b))$ for definable equivalence relations $\varphi(u,v)$ and parameters $a,b$ from $M$ (i.e., these $\varphi_{a,b}$ separate points of $H(G)$).
  Equivalently, $[x,y]_H$ is determined by the values $x(a)=y(b)$ for parameters $a,b\in M^\eq$. We consider the map
$$[x,y]_H\in H(G)\mapsto x^{-1}\circ y\in \Xi,$$
where, on the right, $x,y$ are seen as elements of $\Xi$. By our description of $H(G)$, this is well-defined and injective, and it is clearly a continuous $G$-map. Since $H(G)$ is compact, its image is~$P(M)$.
\end{proof}

\subsection{Characterization of Eberlein pro-oligomorphic groups}\label{Eberlein pro-oligomorphic}

A corollary of the previous results is that if a pro-oligomorphic group $G$ is Eberlein (that is, if we have $W(G)=H(G)$), then $W(G)$ must be an inverse semigroup. As it turns out, this is a sufficient condition. Moreover, this property is related to the following model-theoretic notion.

\begin{defin}
Let $M$ be a structure. We will say that $M$ is \emph{one-based for stable independence} if for any algebraically closed sets $A,B\subset M^\eq$ we have $$A\ind_{A\cap B}B.$$ Equivalently: if for any tuple $a$ and set $B$ we have $\Cb(a/B)\subset\acl(a)$.
\end{defin}

\begin{theorem}\label{Main}
Let $G=\Aut(M)$ be a pro-oligomorphic group. The following are equivalent:
\begin{enumerate}
\item $W(G)$ is an inverse semigroup.
\item The idempotents of $W(G)$ commute.
\item $M$ is one-based for stable independence.
\item $G$ is Eberlein.
\end{enumerate}
\end{theorem}
\begin{proof}
$(1)\Rightarrow (2)$ is just a consequence of the general characterization referred in Fact \ref{equivalence for inverse semigroups}.

$(2)\Rightarrow (1)$: Let $p\in W(G)$, say $p=[x,y]$ for $x,y\in E(M)$. Identifying $x \in E(M)$ with $[1,x] \in W(G)$, we may write $p=[x,y]=[x,1][1,y]=x^*y$. Now, for any $z\in E(M)$ we have $z^*z=1$, so the element $zz^*$ is an idempotent. If idempotents commute, we obtain $pp^*p=x^*yy^*xx^*y=x^*xx^*yy^*y=x^*y=p$.

$(1)\Rightarrow (3)$: By hypothesis, every element is regular, so by Lemma \ref{idempotents and regular elements} we have $x\ind_{x\cap y}y$ for any $x,y\in E(M)$. Now take algebraically closed sets $A,B\subset M^\eq$. By replacing $AB$ by an equivalent copy if necessary, we can find $x\in E(M)$ such that $A\subset x$ and $x\ind_AB$. Again, by replacing $xAB$ by an equivalent copy, we can find $y\in E(M)$ such that $B\subset y$ and $x\ind_By$. In particular, $x\cap y = x\cap B = A\cap B$. Since $x\ind_{x\cap y}y$ and $x\cap y = A\cap B$, we have $x\ind_{A\cap B}y$. Hence $A\ind_{A\cap B}B$.

$(3)\Rightarrow (4)$: We want to show that the canonical map $W(G)\to H(G)$ is injective. Given $p,q\in W(G)$, we can always choose $x,y,z\in E(M)$ such that $p=[x,y]$ and $q=[x,z]$. If the images of $p$ and $q$ in $H(G)$ coincide, then $x\cap y=x\cap z=:C$, and moreover $y\equiv^s_Cz$. Since $M$ is one-based for stable independence, $y\ind_Cx$ and $z\ind_Cx$. By stationarity, we get $y\equiv^s_xz$, that is to say, $p=q$.

$(4)\Rightarrow (1)$: Clear from the identification $H(G)=P(M)$.
\end{proof}

\begin{cor}\label{strongly Eberlein}
The following are equivalent:
\begin{enumerate}
\item $G$ is strongly Eberlein.
\item $M$ is $\aleph_0$-stable (i.e., the space of types $S_u(M)$, in any finite variable $u$, is countable).
\item The intersection $\bigcap_{x \in E(M)}E(M)\cdot x$ is non-empty.
\end{enumerate}
Moreover, if the previous conditions hold, the action of $G$ on $S_u(M)$ is oligomorphic.
\end{cor}
\begin{proof}
$(1)\Leftrightarrow (2)$: As mentioned before, $G$ is a WAP group if and only if $M$ is stable. By the previous theorem, $G$ is strongly Eberlein if and only if $M$ is stable and one-based. A classical result of Zilber (see Theorem~5.12 in \cite[Ch.\ 2]{pil96}, and also \cite{bbhSFB14}, Proposition 3.12) states that an $\aleph_0$-categorical stable structure is one-based if and only if it is $\aleph_0$-stable.

$(2)\Leftrightarrow (3)$: We will argue that
\begin{equation} \label{eq:y in cap Ex}
y\in\bigcap_{x\in E(M)}E(M)\cdot x
\end{equation}
if and only if every type over the image $M'=y(M)$ (possibly in countably many variables) is realized in~$M$. In turn, there exists a submodel $M'$ of $M$ with this property if and only if $M$ is $\aleph_0$-stable.

If every such type is realized and $x$ is any element in $E(M)$, there must be a countable subset $Z\subset M$ such that, with the appropriate orderings, $Mx(M)\equiv Zy(M)$; this induces an element $z\in E(M)$ satisfying $y=zx$. Thus, \eqref{eq:y in cap Ex} holds. Conversely, if $p$ is a type over $M'$, then there is a realization $a$ of $p$ in a countable elementary extension $N$ of $M$; by $\aleph_0$-categoricity, there is an isomorphism $j\colon N\to M$. Let $x\in E(M)$ be the composition $x=jy$, so in particular $x(M)=j(M')$. If \eqref{eq:y in cap Ex} holds, then there is $z\in E(M)$ such that $y=zx$. Thus, $aM'\equiv j(a)x(M)\equiv zj(a)zx(M)=zj(a)M'$, i.e., $zj(a)$ realizes $p$ in~$M$.

For the ``moreover'' part of the statement, it suffices to show that $S_u(M)/G$ is finite. We sketch the (standard) argument. To every indiscernible sequence $(a_i)_{i<\omega}\subset M^{|u|}$ we assign its limit type $p\in S_u(M)$. This is a surjective $G$-map. Since $M$ is one-based, the type of an indiscernible sequence $(a_i)_{i<\omega}$ is determined by $\tp(a_0a_1)$; indeed, one-basedness implies $(a_i)_{i>0}$ is \emph{Morley over $\acl(a_0)$} (see \cite{kimSim}, Fact~6.1.2), and then the claim follows by stationarity. By $\aleph_0$-categoricity, there are only finitely many types $\tp(a_0a_1)$.
\end{proof}

\begin{example}
As mentioned before, the group $S_\infty$ of permutations of a countable set $X$ is (strongly) Eberlein; its Roelcke compactification is the semigroup of partial bijections of $X$ \cite{glamegNew}. We can give some new examples. Consider the following oligomorphic groups:
\begin{enumerate}
\item the automorphism group of a dense linear order, $\Aut(\mbQ,<)$;
\item the homeomorphism group of the Cantor space (or, equivalently, the automorphism group of its algebra of clopen sets), $\Homeo(2^\omega)$;
\item the automorphism group of the random graph.
\end{enumerate}
It follows from the results in \cite[\textsection 6]{bentsa} (see also \cite[\textsection 4.2]{iba14}) that for each of these groups (as well as for $S_\infty$) the algebra $\WAP(G)$ is generated by the functions of the form 
\begin{equation}\label{g mapsto a=gb}
g\mapsto (a=gb)
\end{equation}
for elements $a,b$ in the respective structures. Since these are obviously in $\Hilb(G)$, we deduce that these groups are Eberlein. In fact, for any $G=\Aut(M)$ from the above list, we can deduce the stronger result that $W(G)=P'(M)$, where $P'(M)$ denotes the semigroup of partial elementary maps $M\to M$ with relatively algebraically closed domain. (In particular, as is well-known, these structures have weak elimination of imaginaries; cf.\ Remark~\ref{rem:weak elim imag}.) Indeed, since $\WAP(G)$ is generated by the functions \eqref{g mapsto a=gb}, an element $[x,y]\in W(G)$ is determined by the values of $x(a)=y(b)$ for $a,b\in M$; hence, the canonical map $[x,y]\in W(G)\mapsto x^{-1}\circ y\in P'(M)$ is an isomorphism.
\end{example}

\begin{example}
  \label{ex:Hrushovski}
A famous conjecture of Zilber claimed that an $\aleph_0$-categorical stable structure should be $\aleph_0$-stable (equivalently, one-based, or still: not encoding a \emph{pseudoplane}). This was refuted by Hrushovski, who constructed an $\aleph_0$-categorical stable pseudoplane. The details of the construction can be found in \cite{wagnerRelational}. It follows from Theorem \ref{Main} that the automorphism group of this pseudoplane is an oligomorphic WAP group that is not Eberlein. This answers Question~6.10 in \cite{glamegSurvey}.
\end{example}

\begin{example}
The previous example can be used to produce a countable compact dynamical system of finite Cantor--Bendixson rank that is faithfully representable on a reflexive Banach space, but not on a Hilbert space, in the sense of representability defined in the introduction (see \cite{megHilb} for more background). Indeed, let $M$ be Hrushovski's stable pseudoplane, $G=\Aut(M)$, and choose some formula $\varphi(u,v)$ and parameters $a,b$ such that $f\colon g\mapsto\varphi(a,gb)$ is not in $\Hilb(G)$ (recall the discussion at the beginning of Subsection~\ref{Description of H(G)}). Now, the space $S_\varphi(M)$ of $\varphi$-types in the variable $v$, with parameters from $M$, induces a compactification $X$ of $G$ via the map $g\mapsto\tp_\varphi(gb/M)$. Since $f$ belongs to the associated algebra, the dynamical system $G\actson X$ is not Hilbert-representable; but it is reflexively representable, since $\varphi$ is stable. Finally, as is well-known, the space of local types $S_\varphi(M)$ of a stable formula over a countable structure is a countable compact zero-dimensional space of finite Cantor--Bendixson rank (see, for instance, \cite{pil96}, Remark 2.3 and Lemma 3.1).
\end{example}

\enlargethispage{\baselineskip}

\vspace*{-8pt}
\noindent\hrulefill

\section{Hilbert-representable factors}\label{section factors of H(G)}

In this section we extend our analysis to the factors of $H(G)$ and $W(G)$. We start by showing that all factors of $H(G)$ are zero-dimensional.

We recall that if $\pi\colon G\to U(\mcH)$ is a continuous unitary representation, then $\pi$ extends naturally to a continuous homomorphism $\pi\colon H(G)\to\Theta(\mcH)$. Indeed, the closure of $\pi(G)$ in $\Theta(\mcH)$ is a Hilbert-representable semitopological semigroup compactification of $G$, and thus a $G$-factor of $H(G)$ as per Fact~\ref{fact:universality of H(G)}.

\begin{lem}\label{l:ctble-image}
Let $\pi\colon G\to U(\mcH)$ be a continuous unitary representation of a Roelcke precompact Polish group. Let $\eta \in \mcH$ be a vector such that $\pi(V)\eta = \eta$ for some open subgroup $V \leq G$ (i.e., $\pi(v)\eta=\eta$ for all $v\in V$). Then $\pi(H(G))\eta$ is countable.
\end{lem}
\begin{proof}
  First, we may restrict our attention to the separable Hilbert space generated by $\pi(G)\eta$, which we still denote by $\mcH$. As $G$ is Roelcke precompact and $V$ is open, the set of double cosets $V\backslash G/V$ is finite. Since $\eta$ is fixed by $V$, the function $g\mapsto\langle\eta,\pi(g)\eta\rangle$ factors through $V\backslash G/V$, hence the set $$\{\langle\pi(g_1)\eta, \pi(g_2)\eta\rangle : g_1, g_2 \in G\}$$ is finite. By continuity, $\{\langle\pi(p_1)\eta, \pi(p_2)\eta\rangle: p_1, p_2 \in H(G)\}$ is equal to it, and therefore also finite. So
\begin{equation*}
  \inf \{\| \pi(p_1)\eta - \pi(p_2)\eta \| : p_1, p_2 \in H(G), \pi(p_1)\eta \neq \pi(p_2) \eta\} > 0
\end{equation*}
  and the separability of $\mcH$ implies that $\pi(H(G))\eta$ is countable.
\end{proof}

\begin{prop}\label{p:f-ctble-image}
Let $G$ be a Roelcke precompact Polish group and let $f \in C(H(G))$ be a function such that $Vf = f$ for some open subgroup $V \leq G$. Then $f(H(G))$ is countable.
\end{prop}
\begin{proof}
Let $f = \lim_n f_n$, where $f_n(g) = \langle\xi_n, \pi(g)\eta_n\rangle$ for some representation $\pi$ and vectors $\xi_n, \eta_n$ (recall that every function in $\Hilb(G)$ is a limit of matrix coefficients and we can assume that they are from the same representation simply by taking direct sums). First, we may assume that each $\xi_n$ is fixed by $\pi(V)$. Indeed, let $n$ be such that $\|f - f_n\| \leq \epsilon$ and let $\xi_n'$ be the element of minimal norm of $\overline{\text{co}}(\pi(V)\xi_n)$. Note that $\xi_n'$ is fixed by $\pi(V)$ and for every $g \in G$ and $v \in V$,
  \begin{equation*}
    |\langle\xi_n, \pi(g)\eta_n\rangle - \langle\pi(v)\xi_n, \pi(g)\eta_n\rangle| =
    |f_n(g) - vf_n(g)| \leq 2\epsilon,
  \end{equation*}
implying that
  \begin{equation*}
    |\langle\xi_n, \pi(g)\eta_n\rangle - \langle\xi_n', \pi(g)\eta_n\rangle| \leq 2\epsilon
  \end{equation*}
and thus we can replace $\xi_n$ by $\xi_n'$ without losing much.

Next, by replacing $\pi$ with a sum of infinitely many copies of itself and rescaling if necessary, we may assume that $\xi_n = \xi$ for all $n$. Finally, apply Lemma~\ref{l:ctble-image} to obtain that $\pi(H(G))\xi$ is countable and let $E$ be the equivalence relation on $H(G)$ given by $p \mathrel{E} q \iff \pi(p^*) \xi = \pi(q^*) \xi$ (so that $E$ has countably many classes). Now all $f_n$ and $f$ factor through $E$, so, in particular, the image of $f$ is countable.
\end{proof}

\begin{lem}\label{dense subalgebras}
Suppose $G$ is pro-oligomorphic and let $A\subset\WAP(G)$ be a closed subalgebra. Let $A_0\subset A$ be the subalgebra of functions $f$ such that $Vf=f$ for some open subgroup $V\leq G$, and let $A_1\subset A$ be the subalgebra of functions with finite range.
\begin{enumerate}
\item\label{dense Vf=f} If $A$ is left-invariant, then $A_0$ is dense in $A$.
\item\label{dense finite range} If $A$ is bi-invariant, then $A_1$ is dense in $A$.
\end{enumerate}
\end{lem}
\begin{proof}
This follows almost verbatim from the proofs of Proposition 4.7 and Theorem 4.8 in \cite{bentsa}.
\end{proof}

\begin{theorem} \label{th:zero-dim-factor} If $G$ is a pro-oligomorphic group, then every factor of $H(G)$ is zero-dimensional.\end{theorem}
\begin{proof}
  Let $S$ be a factor of $H(G)$. Since $\mcA(S)$ is left-invariant, by Lemma~\ref{dense subalgebras} (\ref{dense Vf=f}), the subalgebra of functions $f \in \mcA(S)$ that are fixed by some open subgroup $V \leq G$ is dense in $\mcA(S)$. By Proposition~\ref{p:f-ctble-image}, those functions have countable range, and, by density, they separate points in $S$. This implies the conclusion of the theorem.
\end{proof}

\begin{question} Is the same true for all factors of $W(G)$?\end{question}

The automorphism group of the dense, countable circular order acts minimally on the circle and this dynamical system is a quotient of the Roelcke compactification of the group. So certainly some hypothesis is necessary to obtain zero-dimensionality.

The previous theorem, restated as follows, is useful to show that Hilbert-representability is preserved under factors.

\begin{cor}\label{cor:A_1 dense} Let $A\subset\Hilb(G)$ be a left-invariant closed subalgebra, and let $A_1\subset A$ be the subalgebra of functions with finite range. Then $A_1$ is dense in $A$.\end{cor}
\begin{proof}
  Let $S$ be the factor of $H(G)$ corresponding to $A$, so that $A = \mcA(S)\cong C(S)$. By Theorem~\ref{th:zero-dim-factor}, $S$ is zero-dimensional, so functions in $C(S)$ with finite range separate points in $S$. By Stone--Weierstrass, they are dense in $C(S)$.
\end{proof}

\begin{prop}\label{finite range are in B(G)}
Let $G$ be a pro-oligomorphic group. If $f\in\Hilb(G)$ has finite range, then $f\in B(G)$.
\end{prop}
\begin{proof}
Suppose first that $f$ is $\{0,1\}$-valued. By Proposition \ref{Hilb by formulas}, we know that $f$ can be approximated in norm by a linear combination of $\{0,1\}$-valued matrix coefficients $m_0,\dots,m_{n-1}\in B(G)$, say $$\big\|f-\sum_{i<n}\lambda_im_i\big\|<1/2.$$ Hence there is a Boolean function $b \colon \{0,1\}^n\to\{0,1\}$ such that $b\big( (m_i(g))_{i<n} \big) = f(g)$ for every $g\in G$. This implies that $f$ can be written as a Boolean combination of the matrix coefficients $m_i$. Now it is enough to note that, first, the negation of a $\{0,1\}$-valued function $m\in B(G)$ is again in $B(G)$, since we can write it as the difference $\neg m=1-m$, and, second, the conjunction of two $\{0,1\}$-valued functions $m_0,m_1\in B(G)$ is again in $B(G)$, since it is simply the product $m_0\wedge m_1=m_0m_1$. We conclude that $f$ is a matrix coefficient.

Finally, every $f\in\Hilb(G)$ with finite range is a linear combination of $\{0,1\}$-valued functions in $\Hilb(G)$. Indeed, if the range of $f$ is $\{\lambda_1, \ldots, \lambda_k\}$ and $A_i = f^{-1}(\{\lambda_i\})$, we have that $f = \sum_i \lambda_i \mathbf{1}_{A_i}$, and each $\mathbf{1}_{A_i}$ can be written as $\mathbf{1}_{A_i}=F_i\circ f$ for some continuous function $F_i\colon\mbC\to\mbC$. If $f\in\Hilb(G)$, then $F_i\circ f\in\Hilb(G)$ as well (both $f$ and $F_i\circ f$ factor through $H(G)$). Therefore $\mathbf{1}_{A_i}\in B(G)$ and $f\in B(G)$.
\end{proof}

We can finally give an answer to Question \ref{question factors of H(G)} for pro-oligomorphic groups.

\begin{theorem}\label{theorem factors of H(G)}
Let $G$ be a pro-oligomorphic group. Every factor of $H(G)$ is Hilbert-representable.
\end{theorem}
\begin{proof}
  Let $A$ be a subalgebra of $\Hilb(G)$ and let $A_1\subset A$ be the subalgebra of functions with finite range. By Corollary~\ref{cor:A_1 dense} and Proposition \ref{finite range are in B(G)}, we have that $A_1$ is dense in $A$ and contained in $A\cap B(G)$. Hence $A=\overline{A\cap B(G)}$ and, by Proposition~\ref{Equiv Hilb-rep}, the factor of $H(G)$ corresponding to $A$ is Hilbert-representable.
\end{proof}

\begin{cor}
Every semitopological semigroup factor of $H(G)$ is $*$-closed and is an inverse semigroup.
\end{cor}
\begin{proof}
The first claim follows from the previous theorem and Proposition \ref{Hilb rep is *-closed}. Now, any such $G$-factor $H(G)\to S$ must preserve the involution. It follows that $pp^*p=p$ for every $p\in S$, hence $S$ is an inverse semigroup.
\end{proof}

It turns out that the converse of the above corollary also holds. The following is a generalization of Theorem~\ref{Main}.
\begin{theorem}
Let $S$ be a semitopological $*$-semigroup compactification of a pro-oligomorphic group~$G$. The following are equivalent:
\begin{enumerate}
\item $S$ is an inverse semigroup;
\item the idempotents of $S$ commute;
\item $S$ is Hilbert-representable.
\end{enumerate}
\end{theorem}
\begin{proof}
We recall that $S$ is a factor of the WAP compactification, as per Remark~\ref{rem:semitop semigp}(\ref{rem:semitop semigp.bi-inv}) and Fact~\ref{factors of WAP}. The equivalence $(1)\Leftrightarrow (2)$ is then proved exactly as in Theorem \ref{Main}. The implication $(3)\Rightarrow (1)$ is clear, for example by the previous corollary.

$(1)\Rightarrow (3)$: Let $A\subset\WAP(G)$ be the closed algebra generated by the union of $\mcA(S)$ and $\Hilb(G)$, and let $S_H$ be the compactification of $G$ associated to $A$. By Theorem~\ref{theorem factors of H(G)}, to prove $(3)$, it is enough to show that $S_H=H(G)$.

It follows from Facts~\ref{factors of WAP} and~\ref{remark *-closed} that $S_H$ is a semitopological $*$-semigroup compactification. Moreover, since $S$ and $H(G)$ are inverse semigroups, so is $S_H$. To justify this, let $\pi\colon S_H\to S$ denote the canonical factor map, which is a homomorphism of $*$-semigroups; also, recall that we may identify each function $f\in\mcA(S)\subset A$ with its continuous extensions to $S$ and $S_H$. Then, given $q\in S_H$, we have $f(q)=f(\pi(q))=f(\pi(qq^*q))=f(qq^*q)$, since the equality $\pi(q)=\pi(qq^*q)$ holds in~$S$. Similarly, $f(q)=f(qq^*q)$ for each $f\in\Hilb(G)\subset A$. Now, since $S_H$ is the Gelfand space of the closed algebra generated by $A_0=\mcA(S)\cup\Hilb(G)$, the set $A_0$ (seen as a subset of $C(S_H)$) separates points of $S_H$. Hence $q=qq^*q$, and this proves our claim that $S_H$ is an inverse semigroup.

Let $\phi_0\colon W(G)\to S_H$ and $\phi_1\colon S_H\to H(G)$ be the canonical factor maps. We need to show that $\phi_1$ is injective, so that $S_H=H(G)$. The proof of Theorem~\ref{Main} shows that the canonical factor map $\phi_1\phi_0\colon W(G)\to H(G)$ is injective on the set of regular elements of $W(G)$.

Let $p\in W(G)$ be any element, and let $P$ be the closed subsemigroup of $W(G)$ generated by $p^*p$. Then $P^*=P$, so by \cite{bentsa}, Lemma~3.6, there exists an idempotent $e_p\in P$ such that the set $W(G)e_p$ is contained in every set $W(G)s$ for $s\in P$. Then $W(G)e_p\subset W(G)p^*pe_p\subset W(G)pe_p\subset W(G)e_p$, so $W(G)pe_p=W(G)e_p$. It follows from Fact~\ref{fact:Sp=Sqp etc}(\ref{Sp=Sqp}) that $e_p=p^*pe_p$. We set $\sigma(p)=pe_p$ and we note, using Fact~\ref{fact:Sp=Sqp etc}(\ref{idemp self-adj}), that $\sigma(p)$ is regular.

We claim that $\phi_0(\sigma(p))=\phi_0(p)$. Since $P$ is generated by $p^*p$, there is a sequence $n_i<\omega$ such that $(p^*p)^{n_i}\to e_p$. By continuity, $\phi_0(p(p^*p)^{n_i})\to\phi_0(pe_p)$, but we also have $\phi_0(p(p^*p)^{n_i})=\phi_0(p)$ since $\phi_0$ is a homomorphism and $S_H$ is an inverse semigroup.

Finally, let $q,q'\in S_H$ and suppose $\phi_1(q)=\phi_1(q')$. Let $p,p'\in W(G)$ be any elements with $\phi_0(p)=q$ and $\phi_0(p')=q'$. The associated elements $\sigma(p)$ and $\sigma(p')$ are regular and have the same image in $H(G)$, hence $\sigma(p)=\sigma(p')$. It follows that $q=q'$.
\end{proof}

We point out that, even for the group of integers $G=\mbZ$, a Hilbert-representable semitopological semigroup compactification of $G$ need not be an inverse semigroup. The semigroup considered in \cite{bouetalIdempNotClosed} serves as a counterexample.

To conclude, we give a bound on the complexity of the countable factors of $H(G)$.

\begin{prop}
\label{p:CB-rank}
Let $Z \subset \mcH$ be a countable, weakly compact subset of a Hilbert space $\mcH$. Denote $D(Z) = \{\|\xi - \eta\| : \xi, \eta \in Z\}$ and assume that $D(Z)$ is finite. Then $D(Z') \subsetneq D(Z)$ (here $Z'$ is the Cantor--Bendixson derivative of $Z$). In particular, the Cantor--Bendixson rank of $Z$ is bounded by $|D(Z)|$.
\end{prop}
\begin{proof}
Let $\delta = \max D(Z)$; we show that $\delta \notin D(Z')$. Suppose, towards a contradiction, that $\xi, \eta \in Z'$ are such that $\|\xi - \eta\| = \delta$. By translating $Z$, we can assume that $\xi = 0$. Let $\eta_n \to^w \eta$ with $\eta_n$ distinct elements of $Z$. As $\delta$ is maximal, $D(Z)$ is finite, $\eta_n \to^w \eta$, and the norm is lower semicontinuous in the weak topology, we must have that eventually $\|\eta_n\| = \|\eta\|$. This, in turn, implies that $\eta_n \to \eta$ in norm, which means that $\eta_n$ is eventually constant (again, since $D(Z)$ is finite). This contradicts the choice of the sequence $\eta_n$.
\end{proof}

Recall that a \emph{$G$-ambit} $(X,x_0)$ is just a compactification of $G$ where $x_0$ is the image of $1$. If $X$ is countable, then $x_0$ must be isolated.

\begin{cor}
\label{c:CB-finite}
Let $G$ be a Roelcke precompact group and let $(X, x_0)$ be a Hilbert-representable $G$-ambit. Then $X$ is countable if and only if the stabilizer $G_{x_0}$ is open, and in this case we have
  \begin{equation*}
    \rank X \leq |\{G_{x_0} g G_{x_0} : g \in G\}|.
  \end{equation*}
\end{cor}
\begin{proof}
Follows from Lemma~\ref{l:ctble-image} (and its proof) and Proposition~\ref{p:CB-rank}.
\end{proof}

Note that the previous result is a generalization of the following model-theoretic fact, which follows from one-basedness: in an $\aleph_0$-categorical $\aleph_0$-stable theory, the Morley rank of any finite tuple $a$ is bounded by the number of distinct types $\tp(b/a)$ with $\tp(b)=\tp(a)$.

\begin{question}
Do countable ambits of pro-oligomorphic groups necessarily have finite Cantor--Bendixson rank?
\end{question}

We remark that every countable compactification of a Roelcke precompact Polish group $G$ is a factor of $W(G)$. Indeed, since countable compact systems are \emph{Asplund-representable} (see, for instance, \cite{glamegHereditarily}, Corollary 10.2), this follows from \cite{iba14}, Theorem 2.9.

Akin and Glasner \cite{akigla14} construct countable WAP $\mbZ$-ambits of arbitrarily high rank. But of course, the group $\mbZ$ is not pro-oligomorphic.

\noindent\hrulefill

\bibliographystyle{amsalpha}
\bibliography{biblio}

\end{document}